\documentclass[11pt,draft, reqno]{amsart}

\allowdisplaybreaks

\usepackage[margin=18mm]{geometry}
\usepackage{amssymb, amsmath, amsthm, amscd}

\usepackage[T1]{fontenc}
\usepackage{eucal,mathrsfs,dsfont}
\usepackage{color}

\def\<{\langle}
\def\>{\rangle}
\def\wt{\widetilde}

\renewcommand{\leq}{\leqslant}
\renewcommand{\geq}{\geqslant}
\renewcommand{\le}{\leqslant}
\renewcommand{\ge}{\geqslant}

\def\EE{{\mathcal E}}
\def\FF{{\mathcal F}}
\def\LL{{\mathcal L}}

\newcommand{\E}{\mathcal{E}}
\newcommand{\F}{\mathcal{F}}

\definecolor{mno}{rgb}{0.5,0.1,0.5}

\newcommand{\bP}{\mathds P}
\newcommand{\bE}{\mathds E}

\newcommand{\R}{\mathds R}

\newcommand{\bS}{\mathds S}

\newcommand{\e}{\varepsilon}

\newcommand{\Pp}{\mathds P}
\newcommand{\Ee}{\mathds E}

\newcommand{\I}{\mathds 1}
\newcommand{\w}{\omega}

\newcommand{\Z}{\mathds Z}

\newcommand{\Ss}{ \mathds{S}}

\newtheorem{theorem}{Theorem}[section]
\newtheorem{lemma}[theorem]{Lemma}
\newtheorem{proposition}[theorem]{Proposition}
\newtheorem{corollary}[theorem]{Corollary}

\theoremstyle{definition}

\newtheorem{example}[theorem]{Example}
\newtheorem{remark}[theorem]{Remark}

\numberwithin{equation}{section}

\begin{document}
\allowdisplaybreaks
\title[Homogenization of
symmetric stable-like processes] {\bfseries Homogenization of
symmetric stable-like processes in stationary ergodic medium}
\author{Xin Chen,\quad Zhen-Qing Chen,\quad Takashi Kumagai\quad \hbox{and}\quad Jian Wang}

\date{}

\maketitle

\begin{abstract}
This paper studies homogenization of symmetric non-local
Dirichlet forms with $\alpha$-stable-like jumping kernels in one-parameter stationary ergodic environment. Under suitable conditions, we establish
homogenization results and
identify
the limiting effective Dirichlet forms
 explicitly. The coefficients of the
jumping kernels of Dirichlet forms and symmetrizing measures
are allowed to be
degenerate  and unbounded;
and
the coefficients in the effective Dirichlet forms can be
degenerate.

\medskip

\noindent\textbf{Keywords:} homogenization; symmetric non-local
Dirichlet form; ergodic random medium; $\alpha$-stable-like operator

\medskip

\noindent \textbf{MSC 2010:} 60G51; 60G52; 60J25; 60J75.
\end{abstract}
\allowdisplaybreaks

\section{Introduction and results}\label{section1}

\subsection{Background}

The aim of homogenization theory is to provide the macroscopic rigorous characterizations of  microscopically heterogeneous media, which usually involve rapidly oscillating functions of the form $a=a(x/\e)$ with $\e$ being a small positive parameter that characters the microscopic length scale of the media. Homogenization
has been a very active research area for a long time, and
 there is now a vast literature on this topic, see e.g. \cite{A, BLP, JKO, T}.

Homogenization problems for
 random structures are widely studied.
The first rigorous result for
second order
elliptic operators in divergence
forms with stochastically homogeneous
random coefficients
was independently obtained by Kozlov \cite{Koz} and by Papanicolaou and Varadhan
\cite{PV}. The crucial point of
their approaches
is to construct the so-called corrector (i.e., the solution of
certain associated elliptic equations)
and prove that it grows sub-linearly.  Later on a lot of homogenization problems  were investigated  for various elliptic and parabolic differential equations as well as system of equations
in
random stationary media.
In particular, Bourgeat et al.\ in  \cite{BMW} introduced the stochastic version of the two-scale convergence method.
Caffarelli, Souganidis and Wang in \cite{CSW} studied the stochastic homogenization in the context of fully non-linear
uniformly elliptic equations in stationary ergodic environment.

The  goal of this paper  is to address  homogenization problem
for
 non-local operators with random coefficients and to give a characterization of the homogenized
limiting operators.
We start with a brief review of some recent work on
homogenization problems
for non-local operators.
Piatnitski and Zhizhina in \cite{PZ}
studied homogenization problem
for integral operators of convolution type with dispersal kernels (or jumping kernels) that have random stationary ergodic coefficients and
finite second moment. For discrete operators with ergodic weights associated with unbounded-range of jumps (but still having finite second moment),
we refer to \cite{FHS}. In these two cases the scaling order
is naturally a
Brownian scaling, and the limiting operator is a Laplacian.
 One key
element in their approaches
is that the corrector method or the two-scale convergence approach
works when the jumping kernel has finite second moments. However,
when the jumping kernel has infinite second
moment,
the scaling order and limiting process are completely different.
Chen, Kim and Kumagai in \cite{CKK} proved the Mosco convergence of non-local Dirichlet forms associated with
symmetric stable-like random walks in independent long-range
conductance model.
They showed that the
limiting process
is a symmetric $\alpha$-stable L\'evy process.
Kassmann, Piatnitski and Zhizhina in \cite{KPZ} investigated
homogenization of a class of symmetric stable-like processes in ergodic environment whose jumping kernels are of  product form.
Homogenization
problem of symmetric stable-like processes in
two-parameter ergodic environment was also studied in \cite{KPZ}.
We shall mention that in \cite{KPZ} random coefficients of the  jumping kernel are
assumed to be uniformly
elliptic and bounded.
Recently,  Flegel and Heida \cite{FH} considered the corresponding problem in the discrete setting under some moment
conditions on coefficients in two-parameter ergodic
environment. The stochastic homogenization of a class of fully
non-linear integral-differential equations in ergodic environment was
studied by Schwab  \cite{S}. We refer the reader to
\cite{BGG,BRS,PZ1,S0} and references therein for
 homogenization of integral equations (or jump processes) with
periodic coefficients.

As mentioned above, known results concerning
stochastic homogenization of stable-like Dirichlet forms in  one-parameter ergodic environment
 require the coefficients enjoying very special forms
(for examples, the product form). The contribution of this paper is to systematically  study
homogenization problem for symmetric non-local operators in one-parameter ergodic environment under more general settings,
where the corresponding random coefficients can be degenerate and unbounded.

\subsection{Setting}
 Let $d\geq 1$, and $(\Omega, \mathcal{F}, \Pp)$ be the probability space that describes  the random environment on which a measurable group of transformations $\{\tau_x\}_{x\in \R^d}$
 is defined with $\tau_0={\rm id}$,  the identity map on $\Omega$,
 and $\tau_x \circ \tau_x=\tau_{x+y}$  for every $x, y\in \R^d$.
One may think of $\tau_x \w:=\tau_x(\w)$ as a translation of the environment $\w\in \Omega$ in the direction $x\in \R^d$.
We assume that $\{\tau_x\}_{x\in \R^d}$ is stationary and ergodic;
that is,
\begin{itemize}
\item[(i)]  $\Pp(\tau_x A)=\Pp(A)$ for all $A\in \mathcal{F}$ and $x\in \R^d$;

\item[(ii)] if $A\in \mathcal{F}$ and $\tau_xA=A$ for all $x\in \R^d$, then $\Pp(A)\in \{0,1\}$;

\item[(iii)] the function $(x,\w)\mapsto \tau_x \w$ is $\mathscr{B}(\R^d)\times \mathcal{F}$-measurable.
\end{itemize}

Consider
a random variable $\mu: \Omega\rightarrow [0,\infty)$
so that for every $\w \in \Omega$, $\mu (\tau_x \omega) >0$ for a.e.\ $x\in  \R^d$ and $\Ee [\mu ]=1$,
and  a random function $\kappa:\R^d\times \R^d \times \Omega \rightarrow [0,\infty)$ so that for every $\w\in \Omega$,
\begin{equation}\label{eq:shifterc}
\kappa (x,y;\w)=\kappa (y,x;\w), \quad  \kappa (x+z,y+z;\w)=\kappa (x,y;\tau_z\w) \quad
\hbox{for  any }  x,y,z\in \R^d,
\end{equation}
 and
\begin{equation}\label{eq:domain} x\mapsto \int (1\wedge |z|^2)\frac{\kappa (x,x+z;\w)}{|z|^{d+\alpha}}\,dz \in L^1_{loc}(\R^d;dx),
\quad \mbox{$\Pp$-a.e.}
\end{equation}
For each $\w \in \Omega $, these two functions determine a regular symmetric
Dirichlet form $(\EE^\w, \FF^w)$ on
\break \hfill
$L^2(\R^d; \mu(\tau_x \w)\, dx)$ as follows.

Denote by $\Delta:=\{(x,x)\in \R^d\}$ the diagonal of $\R^d \times \R^d$
and $\mu^\w(dx) :=
\mu(\tau_x\w)\,dx$,
which has full support on $\R^d$.
Let $\Gamma$ be an
infinite  cone in $\R^d$ having non-empty interior
that is symmetric with respect to the origin; that is, $\Gamma$ is a non-empty open subset of $\R^d$
so that  $rx\in \Gamma$ for every $x\in \Gamma$ and $r\in \R$.
For  $\alpha\in (0,2)$, define
\begin{equation}\label{e:df}\EE^\w(f,g):=\frac{1}{2}
\iint_{\R^d\times\R^d\setminus \Delta}
(f(x)-f(y))(g(x)-g(y))\frac{\kappa (x,y;\w)}{|x-y|^{d+\alpha}}
\I_{\{ y-x\in \Gamma\}}\,dx\,dy,\quad f,g\in \FF^\w,
\end{equation}
where $\FF^\w$
 the closure of $C_c^1(\R^d)$ with respect to the norm $(\E^\w(\cdot,\cdot)+\|\cdot\|_{L^2(\R^d; \mu^\w (dx))}^2)^{1/2}.$
Note that under \eqref{eq:domain}, $\E^\w(f,f)<\infty$ for all $f\in C_c^1 (\R^d)$. Here and in what follows,
$C^1_c(\R^d)$ (respectively, $C_c(\R^d)$) denotes the space of $C^1$-smooth (respectively, continuous)
functions on $\R^d$ with compact support.
Clearly,   $(\EE^\w, \FF^w)$ is a regular symmetric Dirichlet form on $L^2(\R^d; \mu^\w(dx))$.
So there exist a  Borel subset
$\mathcal{N}^\w\subset\R^d$ having zero
$\EE^\w$-capacity, and a symmetric Hunt process
$X^\w:=\left\{ X^\w_t, {t\ge 0};   \Pp^x, x \in \R^d\setminus\mathcal{N}^\w \right\}$ on
 the state space $\R^d\setminus\mathcal{N}^\w$; see \cite[Chapter 7]{FOT}.
 Note that $X^\w$ is a time change of the process corresponding to
the Dirichlet form $(\E^\w,\F^\w)$ on $L^2(\R^d; dx)$.  When $\Gamma=\R^d$ and $\kappa (x, y; \w)$ is bounded between two positive constants,
this Hunt process is
a symmetric $\alpha$-stable-like process studied in \cite{CK}.

For any $\varepsilon>0$, set $X^{\varepsilon,\w}=(X_t^{\varepsilon,\w})_{t\ge0}:=(\varepsilon X_{t/\varepsilon^{\alpha}}^\w)_{t\ge0}$. The following simple lemma characterizes the
scaled processes $\{ X^{\varepsilon,\w}: \e >0\}$.
 Its proof is postponed to the appendix of this paper.

\begin{lemma}\label{L:d1}
For any $\e>0$, the scaled process $X^{\varepsilon,\w}$ has a
symmetrizing
measure $\mu^{\e,\w}(dx)=\mu (\tau_{x/\e} \w\big)\,dx$, and the associated regular
Dirichlet form  $(\EE^{\varepsilon,\w}, \FF^{\e,\w})$ on $L^2(\R^d;\mu^{\e,\w}(dx))$ is given by
\begin{equation}\label{eq:niebiws}
  \EE^{\varepsilon,\w}(f,g)  =\frac{1}{2}\iint_{\R^d\times \R^d\backslash\Delta}
(f(x)-f(y))(g(x)-g(y)) \frac{\kappa (x/\e,y/\e;
\w)}{|x-y|^{d+\alpha}} \I_{\{ x-y \in \Gamma \}} \,dx\,dy,\quad f,g\in \FF^{\e,\w},
\end{equation}where
$\FF^{\e,\w}$ is the closure of $C^1_c(\R^d)$ with respect to the norm $
\big(\EE^{\e,\w}(\cdot,\cdot)+ \|\cdot\|_{L^2(\R^d;\mu^{\e,\w}(dx))}^2\big)^{1/2}.
$
\end{lemma}

 Let $(\LL^\w, {\rm Dom} (\LL^{ w}))$ and $(\LL^{\e, \w},{\rm Dom} (\LL^{\e, \w})) $ be the $L^2$-generator of the Dirichlet form $(\E^\w,\F^\w)$ on $L^2(\R^d; \mu^{\w})$
 and the Dirichlet form  $(\E^\w,\F^\w)$ on $L^2(\R^d; \mu^{\e, \w})$, respectively.
 That is,
 \begin{equation}\label{e:1.5}
 \LL^\w f(x)= \lim_{\delta \to 0} \frac1{\mu ({\tau_x}  \w) } \int_{ \{y\in \R^d: |y-x|>\delta \}} ( f(y)-f(x))  \frac{\kappa (x, y, \w)}{|y-x|^{d+\alpha}}
 \I_{\{ y-x\in \Gamma \}} \,dy
 \quad \hbox{for } f\in {\rm Dom} (\LL^{ w})
 \end{equation}
 and
 \begin{equation}\label{e:1.6}
 \LL^{\e, \w}  f(x)= \lim_{\delta \to 0} \frac1{\mu ( \tau_{x/\e} \w) }\int_{\{y\in \R^d: |y-x|>\delta\}}
 ( f(y)-f(x)) \frac{\kappa (x/\e, y/\e, \w)}{|y-x|^{d+\alpha}} \I_{\{y-x\in \Gamma \} }\, dy
 \quad \hbox{for } f\in {\rm Dom} (\LL^{\e, \w}).
 \end{equation}
It is easy to see that for each $\e>0$,
$g\in {\rm Dom} (\LL^{\e, \w} )$ if and only of $g^{(\e)}\in {\rm Dom}(\LL^\w)$, where
$g^{(\varepsilon)}(x)=g(\varepsilon x)$, and
$$
\LL^{\e,\w}g(x)= \varepsilon^{-\alpha} \LL^{\w} g^{(\varepsilon)}(x/\varepsilon).
$$

For any $\lambda>0$ and $f\in C_c (\R^d)$, let $u_f^{\e,\w}$ be the solution to
$$
 (\lambda -\LL^{\e,\w} ) u_f^{\e,\w}  =f
$$
 in $L^2(\R^d;\mu^{\e,\w}(dx))$.
The main goal of homogenization problem in our paper is to show, under suitable conditions, that almost surely,
$u_f^{\e,\w}$ converges to a deterministic function $u_f$ as $\e \to 0$
for every $f\in C_c(\R^d)$,
and that $u_f$ is the solution of
$$
 (\lambda -\LL ) u_f =f ,
$$
where $\LL$ is the $L^2$-generator of certain regular symmetric Dirichlet form $(\E,\F)$ on $L^2(\R^d;dx)$ whose jumping kernel can be degenerate.
The convergence of
$u_f^{\e,\w}\to u_f$ as $\e\to0$ is in the resolvent topology; that is,
for a.s.\ $\w\in \Omega$,
$$
\lim_{\e\to0}\|u_f^{\e,\w}-u_f\|_{L^2(\R^d;\mu^{\e,\w}(dx))}=0.
$$
See \cite{MT, T} for background and \cite{BGG, BRS, KPZ} for recent study
 on homogenization problems related to non-local operators.

Let $ K(z)$ be a non-negative bounded
measurable even function on $\R^d$ (that is, $K(z)=K(-z)$ for all $z\in \R^d)$.
Define a regular Dirichlet form $(\E^K,\F^K)$ on $L^2(\R^d;dx)$
by
\begin{equation}\label{eq:DeDFK}\begin{split}
{ \EE}^{K}(f,g)&=\frac{1}{2}\iint_{\R^d\times\R^d\setminus \Delta} (f(x)-f(y))(g(x)-g(y))
\frac{K(x-y)}{|x-y|^{d+\alpha}}
\I_{\{ x-y\in \Gamma\}}
\,dx\,dy,\quad f,g\in \F^K,
\end{split}\end{equation}
where $\F^K$ is
 the closure of $C_c^1(\R^d)$ in  with respect to the norm
$\big(\E^K(\cdot,\cdot)+\|\cdot\|_{L^2(\R^d;dx)}^2\big)^{1/2}.$
The limiting Dirichlet form $(\E,\F)$ for homogenization problem considered in this paper is of this type. We emphasis that
the symmetric
cone
$\Gamma$ in \eqref{eq:niebiws} can be a proper
subset of $\R^d$
in our paper.

\subsection{Main results}
Our main results are divided into two cases, according to the explicit form of the coefficient $\kappa(x,y;\w)$ in \eqref{e:df}. The first one is concerned on the case that $\kappa(x,y;\w)$ is of the summation form, and
the second one on the case that
is of a
 product form.

\subsubsection{{\bf$\kappa(x,y;\w)$ of  summation form}}

To state the statement of this part, we need the
following assumption
${\bf(A)}$ for $\kappa (x,y;\omega)$.

\medskip

\noindent {\it{\bf (A1)} For a.s.\ $\w\in \Omega$,
\begin{equation}\label{assu-1}
\kappa (x,y;\w)=\nu(y-x;\tau_x\w)+\nu(x-y;\tau_y\w)\quad \hbox{for every  } x,y\in \R^d,
\end{equation}
where $\nu:\R^d\times \Omega \rightarrow [0,\infty)$ is a measurable random function
such that
\begin{itemize}
\item[(i)]
There is a constant $l>0$ such that for any $n>0$ and $x, z_1,z_2\in \R^d$,
\begin{equation}\label{a2-1-1}\begin{split}
|{\rm Cov} \left(\nu_n(z_1; \cdot), \nu_n(z_2; \tau_x (\cdot))\right)|
:&= \big|\Ee\left[ \nu_n(z_1;\cdot)\cdot\nu_n(z_2;\tau_x(\cdot))\right]-\Ee[
\nu_n(z_1;\cdot)] \Ee [\nu_n(z_2;\cdot)] \big|\\
&\le
C_1(n)
\big(1\wedge |x|^{-l}\big),
\end{split}
\end{equation}
where $\nu_n=\nu\wedge n$ and $C_1(n)$ is a positive constant depending on $n$.
\item[(ii)] There is a non-negative
 measurable
function $\bar \nu$ on $\R^d$ such that $\Ee[\nu({z}/\e;\cdot)]$ converges weakly to $\bar \nu(z)$ in $L_{loc}^1(\R^d;dx)$ as $\w\to0$; that is,
\begin{equation}\label{a2-1-2a}
\lim_{\e\to0}\int_{\R^d}h(z)\Ee[\nu({z}/\e;\cdot)]\,dz=\int_{\R^d} h(z)\bar \nu(z)\,dz\quad \hbox{for every  } h\in L_{loc}^\infty(\R^d;dx).
\end{equation}\end{itemize}

\medskip

\noindent{\bf(A2)} There are non-negative random variables $\Lambda_1\le
\Lambda_2$ on $(\Omega, \mathcal{F}, \Pp)$ with
\begin{equation}\label{t2-1-1}
 \Ee[ \Lambda_1^{-1}+\Lambda_2^{p}
] <\infty,
 \end{equation}
 for some constant  $p>1$ so that
for a.s.\ $\w\in
\Omega$,
\begin{equation}\label{a2-2-1}
 \Lambda_1(\tau_x\w)+\Lambda_1(\tau_y \w)\le \kappa (x,y;\w)\le \Lambda_2(\tau_x\w)+\Lambda_2(\tau_y\w)\quad \hbox{for every  } x,y\in \R^d.
 \end{equation}

 \medskip

\noindent{\bf(A2')} There are non-negative random variables $\Lambda_1\le
\Lambda_2$ on $(\Omega, \mathcal{F}, \Pp)$ so that \eqref{a2-2-1} holds for a.s.\ $\w\in \Omega$, and that
$$
 \Ee[ \Lambda_1^{-1}] <\infty.
$$}

It is obvious that condition {\bf (A2')} is weaker than condition {\bf (A2)}.
Here are
some comments on assumption {\bf(A)}.

\begin{remark}\label{R:1.2}  \rm \begin{itemize}
\item[(i)] It is easy to see that any $\kappa(x, y; \w)$ of form \eqref{assu-1} enjoys the property \eqref{eq:shifterc}.
On the contrary, any $\kappa (x, y; \w)$ satisfying  \eqref{eq:shifterc} admits a representation of the form
\eqref{assu-1}.
This is because   $\kappa (x, y; \w)= \kappa (0, y-x; \tau_x \w)$ and so by the symmetry
of $\kappa (x, y; \w)$ in $(x, y)$ we have
$$
\kappa (x, y; \w)=\tfrac12 ( \kappa (x, y;\w)+ \kappa (y, x; \w)) =\tfrac12 ( \kappa (0, y-x; \tau_x \w) + \kappa (0, y-x; \tau_y \w)).
$$
Hence we can write $\kappa (x, y; \w)$ as
$$
\kappa (x, y; \w) = \nu(y-x; \tau_x \w) + \nu (x-y; \tau_y \w),
$$
where
\begin{equation}\label{e:1.15}
\nu(x;\w):=  \kappa (0, x; \w) /2 .
\end{equation}
Thus \eqref{eq:shifterc} is a symmetrized and long-range analogy of
nearest neighborhood random walk models with balanced random conductance;
 see e.g. \cite{BD, DG,DGR,GZ}.
 Representing  $\kappa (x, y; \w)$  via \eqref{assu-1} by a general $\nu(z; \w)$ rather than that of \eqref{e:1.15}
 allows more flexibility in satisfying the mixing condition \eqref{a2-1-1} on $\nu_n=\nu\wedge n$.

\item[(ii)]
Unlike elliptic differential operators, we have a
variable $(y-x)/\e$ by shifting operators $\tau_{x/\e}$ and $\tau_{y/\e}$  in the coefficient
$$
\kappa (x/\e, y/\e; \w)= \kappa (0, (y-x)/\e; \tau_{x/\e} \w) = \kappa (0, (x-y)/\e; \tau_{y/\e} \w)
$$
 of the  scaled
process $X^\e$ which corresponds to the long
range property of the jumping kernel (see \eqref{eq:niebiws}).
This prevents
us to directly applying the ergodic theorem to deduce the almost sure
convergence as indicated below. We thus assume some
kind of mixing condition \eqref{a2-1-1} on $\nu_n$, uniformly
in $z_1, z_2\in \R^d$,
to guarantee this convergence.
Similar assumption
(without the variable
$z_i \in \R^d$ and on $\nu$ itself)
 has been used in \cite[Assumption A3]{JR} to establish the
quenched
functional central limit
theorem for random walks on $\R^d$ where
the random
environment is i.i.d.\
in time and polynomially mixing in space, and in \cite[Assumption
A5]{An}
in the study of
 invariance principle for diffusions in time-space
ergodic random environment. We mention that \eqref{a2-1-1} includes the
so-called \lq\lq unit range of dependence\rq\rq\  condition
used
in \cite[p.\ xii, (0.6)]{AMK}.

\item[(iii)] Suppose that \eqref{a2-1-1} holds with $C^*_1:=\limsup_{n\to\infty}C_1(n)<\infty$. Then, for any $x,z_1,z_2\in \R^d$,
$$
|{\rm Cov} (\nu(z_1; \cdot), \nu(z_2, \tau_x (\cdot)))|
 \le
C^*_1
\big(1\wedge |x|^{-l}\big).
$$
This, along with the symmetry of
of $\kappa (x, y; \w)$ in $(x, y)$, yields that for every $x_1, y_1, x_2, y_2 \in \R^d$,
 \begin{equation}\label{e:1.13} {\rm Cov} (\kappa (x_1,y_1;\cdot), \kappa (x_2,y_2;\cdot) ) \leq  4C^*_1 \Big(1\wedge \big(
 |x_1-x_2|\wedge|x_1-y_2|\wedge|y_1-x_2|\wedge|y_1-y_2|\big)^{-l} \Big).\!\!\!\!\!\!\!\!\!\!
 \end{equation}
 Thus, the mixing condition \eqref{e:1.13} is
weaker than the
mutually independent stable-like random conductance models
investigated in \cite{CKK,CKW}. In details, \eqref{e:1.13} only requires
the mixing condition on the position variable $x$, not on the jumping size variable $z$; while in \cite{CKK,CKW}
the mutual independence is
imposed
 on both variables $x$ and $z$, which was crucial  to verify
{(A4*) (ii)} in \cite{CKK} (see also \cite[Section 4]{CKW}). In some sense the mutually independent assumption adopted in \cite{CKK,CKW}
corresponds
to the following mixing condition:
there are constants $l,C_2>0$ so that for every $x_1,x_2,y_1,y_2\in \R^d$,
$$
{\rm Cov} (\kappa (x_1,y_1;\cdot), \kappa (x_2,y_2;\cdot) )
\le  C_2 \Big(1\wedge \big(|x_1-x_2|+|y_1-y_2|)\wedge (|x_1-y_2|+|x_2-y_1|)\big)^{-l}\Big)
$$
which is stronger than \eqref{e:1.13}.
Since in the present paper
the mixing condition is of the weaker form \eqref{a2-1-1} and acting on $\nu_n$ instead of $\nu$,
the arguments in \cite{CKK,CKW} does not work.
We use a different approach in this paper
to deal with the homogenization problem.

\item[(iv)] It follows from \eqref{a2-1-2a} that $\bar \nu (z)$ is a radial process; that is, $\bar \nu (\lambda z)=\bar \nu  (z)$ for any $\lambda >0$.
Moreover, under condition {\bf(A2)}, there are positive constants $C_3\leq C_4$ so that
$$  C_3\leq \bar \nu(z) +\bar \nu(-z) \leq C_4 \quad \hbox{for all } z\in \R^d.
$$ Note also that in our setting we always assume that \eqref{eq:domain} holds true.
In fact,  \eqref{eq:domain} is a consequence of assumption {\bf (A2)}. Indeed, suppose  {\bf (A2)} holds.
Then by the Fubini theorem, for any $R\geq 1$,
\begin{align*}
  \quad\quad \bE \left( \int_{B(0, R)}\! \int_{\R^d} (1\wedge |z|^2) \frac{\kappa (x, x+z; \w )}{|z|^{d+\alpha}}\, dz \, dx\right)
&\leq \! \int_{B(0, R)}  \!\int_{\R^d} (1\wedge |z|^2) \frac{\bE [ \Lambda_2 (\tau_x (\cdot) )] + \Ee[\Lambda_2 (\tau_{x+z}(\cdot)))]}{|z|^{d+\alpha}} \,dz \, dx\\
&\leq  2\bE [ \Lambda_2 ] \int_{B(0, R)}  \int_{\R^d} \frac{1\wedge |z|^2}{|z|^{d+\alpha}}\, dz \, dx <\infty.
\end{align*}
In particular, we have $\bP$-a.s.,  $$\int_{B(0, R)} \int_{\R^d} (1\wedge |z|^2) \frac{\kappa (x, x+z; \w )}{|z|^{d+\alpha}} \,dz \,dx <\infty$$
for every $R>0$.
\end{itemize}
\end{remark}

\begin{theorem}\label{thm:mainth1}
Suppose that {\bf (A1)} and {\bf (A2)}
hold, and that $\bE [ \mu^p]<\infty$ for some $p>1$.
 For $\e>0$, let $U_\lambda^{\e,\w}$ be the $\lambda$-order resolvent of the Dirichlet form $(\EE^{\e,\w}, \FF^{\e,\w})$ given by \eqref{eq:niebiws}.
There is $\Omega_0\subset \Omega$ of full probability measure  so that for every $\w \in \Omega_0$,  every $\lambda>0$ and $f\in C_c (\R^d)$,
$$U_\lambda^{\e,\w}f \hbox{ converges to } U^K_\lambda f \hbox{ in }L^1(B(0,r); dx)\quad\hbox{   as } \e\to0 $$ for any $r>1$, and
\begin{equation}\label{e:1.17a}
\lim_{\e \to 0} \| U_\lambda^{\e,\w}f - U^K_\lambda f\|_{L^2(\R^d; \mu^{\e, \w} )}=0,
\end{equation}
where $U^K_\lambda $ is the $\lambda$-order resolvent of the symmetric Dirichlet form $({ \EE}^{K},\F^K)$
on $L^2(\R^d; dx)$ given by \eqref{eq:DeDFK} with
$$
K(z) :=\bar \nu(z)+\bar \nu(-z).
$$
\end{theorem}

Clearly, by taking the smaller one, we can   assume  $p>1$ in the condition $\bE [ \mu^p]<\infty$ is the same as the $p>1$ in {\bf (A2)}.
Note that by Remark \ref{R:1.2}(iv), $K(z)$ is a    radial even function on $\R^d$ that is bounded between two positive constants.
So the limiting Dirichlet form $({ \EE}^{K},\F^K)$ is that of a symmetric, but not necessary rotationally symmetric
 $\alpha$-stable process on $\R^d$. Since for any $g\in C^1_c(\R^d)$,
 \begin{align*}
 &\EE^{\e, \w} (U_\lambda^{\e,\w}f , g)+ \lambda \< U_\lambda^{\e,\w}f , g\>_{L^2(\R^d; \mu^{\e, \w}(dx))} =\< f, g\>_{L^2(\R^d; \mu^{\e, \w}(dx))} ,\\
  & \EE^{K} (U^K_\lambda f , g)+ \lambda \< U_\lambda^{K}f , g\>_{L^2(\R^d;  dx)} =\< f, g\>_{L^2(\R^d;  dx)},
 \end{align*}
 and by the Birkhoff ergodic theorem (see Proposition \ref{P:2.1} below),
 $$
 \lim_{\e\to 0} \< U^K_\lambda f, g\>_{L^2(\R^d; \mu^{\e, \w}(dx))} = \< U^K_\lambda f, g\>_{L^2(\R^d; dx)}
  \quad \hbox{and} \quad
 \lim_{\e\to 0} \< f, g\>_{L^2(\R^d; \mu^{\e, \w}(dx))} = \< f, g\>_{L^2(\R^d; dx)} ,
  $$
 we conclude from \eqref{e:1.17a} that
 $$
 \lim_{\e\to 0} \EE^{\e, \w} (U_\lambda^{\e,\w}f , g) = \EE^{K} (U^K_\lambda f , g)
 \quad \hbox{for every } g\in C^1_c(\R^d).
 $$

\subsubsection{{\bf$\kappa(x,y;\w)$ of product form}}

Motivated by \cite[(Q1)]{KPZ}, we next consider the
the case where  the coefficient $\kappa (x,y;\w)$ of \eqref{eq:shifterc}
is of  a product form.
(We note that \cite[(Q2)]{KPZ} is essentially an approximation of \cite[(Q1)]{KPZ}, under some additional assumptions on
the random environment $(\Omega,\FF,\Pp)$ and on the coefficient $\kappa (x,y;\w)$ of the jumping kernel.)
We consider the following Assumption ${\bf(B)}$.

\medskip

\noindent  {\bf (B1)}
{\it For a.s. $\w\in \Omega$,
\begin{equation}\label{e:1.17}
\kappa (x,y;\w)=\nu_1(\tau_x\w)\nu_2(\tau_y\w) +\nu_1(\tau_y\w)\nu_2(\tau_x\w)\quad \hbox{for every } x,y\in \R^d,
\end{equation} where $\nu_1$ and $\nu_2$ are non-negative random variables on $(\Omega, \FF,\Pp)$.
}

\noindent  {\bf (B2)}
{\it There are non-negative random variables $\Lambda_1 \leq \Lambda_2$ on $(\Omega, \FF, \bP)$
so that for a.s. $\w\in \Omega$,
\begin{equation}\label{a2-2-2}
\Lambda_1(\tau_x\w)\Lambda_1(\tau_y\w)\le \kappa (x,y;\w) \le \Lambda_2(\tau_x\w)\Lambda_2 (\tau_y\w)
\quad \hbox{for every } x,y\in \R^d,
\end{equation}
and that
\begin{equation}\label{t2-1-**}
 \Ee \left[\Lambda_1^{-1}+\Lambda_2^{2}  \right] <\infty.
\end{equation}
}

\noindent  {\bf (B2')}
{\it There are non-negative random variables $\Lambda_1 \leq \Lambda_2$ on $(\Omega, \FF, \bP)$
so that \eqref{a2-2-2} holds for a.s.\ $\w\in \Omega$, and that
$$
 \Ee \left[\Lambda_1^{-1} \right] <\infty.
$$ }

Clearly condition {\bf (B2)} is stronger than  condition {\bf (B2')}.
Conditions {\bf (B2)} and {\bf (B2')} are equivalent to the following, whose proof is postponed into the appendix of this paper.

\begin{proposition}\label{P:1.4}
Suppose that $\kappa (x, y; \w)$ is given by \eqref{e:1.17} for some non-negative random variables
$\nu_1$ and $\nu_2$ on $(\Omega,\FF,\Pp)$.
Then condition {\bf (B2')} holds if and only if $\Ee \left[  (\nu_1 \nu_2)^{-1/2} \right]<\infty$, and condition {\bf (B2)} holds if and only if
\begin{equation}\label{e:1.16}
\Ee \left[  (\nu_1 \nu_2)^{-1/2} + (\nu_1 + \nu_2)^{2} \right] <\infty.
\end{equation}
\end{proposition}

\begin{remark}\label{remark-add} Following the argument in Remark \ref{R:1.2}(iv) and using the elementary inequality $ab\le (a^2+b^2)/2$ for $a,b\ge0$, we can easily verify that \eqref{eq:domain} is a consequence of assumption {\bf (B2)}.
 Moreover, if \eqref{a2-2-2} holds and the function  $x\mapsto \Lambda_2(\tau_x\w)$ is locally bounded for a.s.\ $\w\in \Omega$,
 we can estahblish    \eqref{eq:domain} just under the first moment condition of $\Lambda_2$ (i.e. under the condition that
  $\Ee[\Lambda_2]<\infty$).
Indeed, for any $R\ge 1$ and a.s.\ $\w\in \Omega$,
\begin{align*}
&\int_{B(0, R)}\! \int_{\R^d} (1\wedge |z|^2) \frac{\kappa (x, x+z; \w )}{|z|^{d+\alpha}}\, dz \, dx\\
&\leq \left[\sup_{x\in B(0,2R)}\Lambda_2(\tau_x\w)^2\right]\cdot \! \int_{B(0, R)}  \!\int_{B(0,R)} \frac{|z|^2}{|z|^{d+\alpha}}\,dz \, dx\\
&\quad +
\sum_{k=[\log R/\log 2]}^\infty 2^{-k(d+\alpha)}\int_{B(0,R)}\int_{\{2^k\le |z|\le 2^{k+1}\}}\Lambda_2(\tau_x\w)
\Lambda_2(\tau_{x+z}\w)\,dz \, dx\\
&\leq  C_1(R;\w)+\sum_{k=[\log R/\log 2]}^\infty  2^{-k(d+\alpha)}\left(\int_{B(0,R)}\Lambda_2(\tau_x\w)\,dx\right)\cdot
\left(\int_{B(0,2^{k+2})}\Lambda_2(\tau_y\w)\,dy\right)\\
&\leq C_1(R;\w)+C_2(R;\w)\sum_{k=[\log R/\log 2]}^\infty 2^{-k(d+\alpha)}2^{(k+2)d}<\infty,
\end{align*}
where in the third inequality we used the Birkhoff ergodic theorem (see Proposition \ref{P:2.1} below) and the last inequality follows from the local boundedness of $\Lambda_2(\tau_x\w)$.
The assumption such as
the local boundedness of $x\mapsto \Lambda_2(\tau_x\w)$ was
 assumed in \cite{CD} (see (a.3) on p.1536)
 to study the invariance principle for symmetric diffusions in a degenerate and unbounded stationary and ergodic random
 medium.
\end{remark}

Note that  any $\kappa(x, y; \w)$ of form \eqref{e:1.17} enjoys the property \eqref{eq:shifterc}.
When the coefficient $\kappa(x, y; \w)$ of the jumping kernel  is of the  product form \eqref{e:1.17},
the corresponding symmetric Dirichlet form  $(\EE^\w, \FF^\w)$ has the expression
$$
\EE^\w(f, f):=\frac{1}{2}
\iint_{\R^d\times\R^d\setminus \Delta}
(f(x)-f(y))^2 \frac{\nu_1 (\tau_x \w)  \nu_2(\tau_y \w) }{|x-y|^{d+\alpha}}
\I_{\{ y-x\in \Gamma\}}\,dx\,dy \quad \hbox{for } f \in \FF^\w.
$$
In this case, we are able to drop the mixing condition
\eqref{a2-1-1} from Theorem \ref{thm:mainth1}.

\begin{theorem}\label{T:1.5}
Suppose that assumptions {\bf (B1)} and {\bf (B2)}
hold, and that $\bE [ \mu^{p } ]<\infty$ for some $p>1$.
Then the conclusion of Theorem $\ref{thm:mainth1}$ holds with constant
$$K(z):= \bE [\nu_1] \, \bE [\nu_2].$$
\end{theorem}

\subsection{Comments on main results}
 To the best of our knowledge, only two cases have been
studied in literature
concerning  homogenization of
$\alpha$-stable-like processes (or $\alpha$-stable-like operators) in  ergodic random environment.
The first one is \cite{KPZ},
where the infinitesimal generator is given by \eqref{e:1.5} with $\Gamma =\R^d$ and
$\kappa (x,y,\w)=\lambda_1  (\tau_x \w)\lambda_2 (\tau_y \w)$ in one-parameter ergodic random environment.
The second one is  \cite{FH}, which is under two-parameters ergodic environment.
The setting of our paper is more general, and
it also
includes the symmetrization of those in \cite{DG,GZ,S}; see {\bf (A1)} introduced above.
Besides, as we will explain in Example \ref{exa2} below, assumption
{\bf (B)}
is also related to
random conductance models associated with mutually independent site
percolations.

 As far as we know, for all the  results in literature, the limiting process is
always
non-degenerate with $\Gamma =\R^d$,
 even if the coefficients of scaled processes are degenerate
-- see for instance \cite{BD,CD,CKW,CKK,KPZ, FH}.
Our paper provides examples for symmetric jump processes that both coefficients of scaled processes
and the limiting  process are degenerate. In details, in our paper
not only the cone $\Gamma$ can be a proper subset of $\R^d$  but also the coefficient
$\kappa (x, y, \w)$ of jumping kernel can be degenerate and unbounded.
Moreover, the coefficient $K(z)$ of jumping kernel for the limiting process
can be a non-constant function, which in particular implies that the limiting process
does not need to be a
rotationally symmetric $\alpha$-stable process, but a more general symmetric $\alpha$-stable L\'evy process on $\R^d$ that enjoying the scaling property.

Under assumption {\bf (A)}, we only assume the finiteness of negative 1-moment and
positive $p$-moment with $p>1$ for bounds of the coefficient $\kappa(x,y;\w)$ to study the homogenization problem. We believe that the negative moment
integrability condition is optimal and the positive moment
integrability condition is almost optimal, since they are necessary to apply the ergodic theorem. We emphasis that under assumption {\bf (B)}, we also only require the finiteness of negative 1-moment.  We also note that, under both negative $1$-moment and positive
$1$-moment conditions, the annealed invariance principle for nearest neighbor random walks on random conductances was established in \cite{DMFGW}, and the quenched invariance principle was proven in \cite{Bis} when $d=1,2$. For random divergence forms,
one may follow the two-scale convergence method adopted in \cite{ZP}
to prove the $L^2$-convergence of associated resolvents under similar conditions.

It is natural to consider further
the weak convergence of the
scaled
processes on
the path space. Strong convergence of the
resolvents that we have established so far corresponds to the
convergence of the finite dimensional distributions of the
scaled  processes
when the initial measure is absolutely continuous with respect to an invariant measure.
In order to obtain the weak convergence of the scaled  processes,
we need to establish the tightness (with respect to the Skorohod topology) of the scaled  processes.
In fact, if the initial distribution is an invariant measure (or
more generally it is absolutely continuous with respect to an
invariant measure), then the tightness can be obtained by using the
so-called forward-backward martingale decomposition (see
\cite[Proposition 3.4]{CKK} for the corresponding statement in the
discrete setting). Hence one can obtain the convergence of the
processes on the path space under such initial condition (or under
some weaker topology), see \cite[Theorems 2.2 and 2.3]{CKK} for more
discussions in the discrete case.
When $(x, y)\mapsto \kappa (x, y; \cdot)$ is bounded between two positive constants,
we can use heat kernel estimates from \cite{CK} when $\Gamma =\R^d$ or
parabolic Harnack inequality from \cite{CKW2}  when $\Gamma\subsetneq \R^d$
to establish the tightness and therefore the weak convergence of the scaled processes
starting from any point.
However, it is highly non-trivial
to prove such convergence if the process starts at any fixed point (in
other word, if the initial distribution is a Dirac measure)
when $(x, y) \mapsto \kappa (x, y;\cdot )$ is not bounded between two positive constants.
We will
address this problem in a separate paper.

\subsection{Example}

A typical example of infinite symmetric cone for degenerate non-local Dirichlet forms given by \eqref{e:df} is
$$
\Gamma = \Big\{ z\in \R^d: \langle z, z_0\rangle \geq \eta |z|\Big\}$$
for some $z_0\in \bS^{d-1}$ and $\eta \in [0, 1)$.
 In the deterministic case, the
regularity
estimates for non-local operators associated with such kind of degenerate Dirichlet forms
have been studied in \cite{DK, CKW2}; see \cite[Example 3]{DK} or \cite[Example 1.2]{CKW2} for more details.

As an application of Theorem \ref{T:1.5}, we take the following example that
improves \cite[Theorem 3, Case (Q1)]{KPZ}, where the coefficients
$\lambda_i( \tau_x \w)$
$(i=1,2)$ are assumed to be
uniformly bounded from above and below and $\Gamma = \R^d$.

\begin{example}\label{exa2}\it
Suppose that $\Gamma$ is an infinite symmetric cone in $\R^d$ that has non-empty interior.
For any $\e>0$, let $\LL^{\e,\w}$ be the  L\'evy-type operator given by
\begin{equation}\label{e:prod-ex}
\LL^{\e,\w} f(x)={\rm p.v.} \int_{\R^d} (f(y)-f(x))\frac{\lambda_1(\tau_{x/\e} \w)\lambda_2( \tau_{y/\e} \w)}{|y-x|^{d+\alpha}}\,
\I_{\Gamma} (y-x)\,dy,
\end{equation}
where $\lambda_1$ and $\lambda_2$ are non-negative measurable functions on $(\Omega,\FF,\Pp)$ such that
$$
  \lambda_2^{-1} \in L^1
  (\Omega;\Pp),\,\,\,\,   \lambda_2  \in L^{2} (\Omega;\Pp)
\quad \hbox{and} \quad  {\lambda_2 }/{\lambda_1}\in L^p(\Omega;\Pp),
$$
for some $p>1$.
Then,
as $\e\to0$,  $\LL^{\e,\w}$ converges in the resolvent topology to
$$\LL f(x)={\rm p.v.} \int_{\R^d}  (f(y)-f(x))\frac{C_0}{|y-x|^{d+\alpha}}\,dy,$$ where
\begin{equation}\label{e:1.18}
C_0= \frac{  \left(\Ee  [\lambda_2 ] \right)^2}{ \Ee \left[{\lambda_2 }/{\lambda_1} \right]}
\end{equation}
in the following sense. There is $\Omega_0\subset \Omega$ of full probability measure  so that for every $\w \in \Omega_0$,
$\lambda >0$ and $f\in C_c(\R^d)$,
$$
\lim_{\e \to 0} \| U_\lambda^{\e,\w}f - U_\lambda f\|_{L^2(\R^d; \bar \mu^{\e; \w}(dx) )}=0,
$$
where $\bar \mu^{\e, \w}(dx):= (\lambda_2/\lambda_1)(\tau_{x/\e} \w)\, dx$, and
 $U^{\e, \w}f$ and $U_\lambda f$ are the $\lambda$-order resolvent function of $\LL^{\e, \w}$ and $\LL$,
respectively.  In addition, $U_\lambda^{\e,\w}f$ converges to $U_\lambda f$ in $L^1(B(0,r);dx)$, as $\e\to0$, for every $r>1$.
 \end{example}

The operator \eqref{e:prod-ex} can be seen as a
 random long range randomly weighted
 site model. Indeed, if
 $\lambda_1 ( \tau_x \omega)$ and $\lambda_2 ( \tau_x \omega)$
 are regarded as  random weight at the site $x$ for initiating a jump and receiving a jump, respectively,
 then the long range effect of the media
 for the coefficient of the jump intensity  from $x$ to $y$ is given by the product
 $\lambda_1 ( \tau_x \omega) \lambda_2 ( \tau_y \omega)$.
  At the first sight, the constant coefficient $C_0$ for the limiting operator $\LL$
  should be $\bE [ \lambda_1 \lambda_2]$,
  but with the idea of the time change as used in the proof
  of the assertion of Example \ref{exa2} below, it turns out  the correct one should be the one given by formula \eqref{e:1.18}.
   We emphasize again  that in this site model the mixing condition
\eqref{a2-1-1}  of the media given in Assumption {\bf(A1)} is not needed.

\subsection{Organization of the paper}
The rest of the paper is organized
 as follows.
 In Section \ref{S:2},
 we will prove homogenization of stable-like Dirichlet forms under general sufficient conditions. The main results are Theorems \ref{T:2.2} and \ref{C:note3}.
In Section
\ref{S:3},  we study the weak convergence of non-local symmetric  bilinear forms,
and
the $L^1$-precompactness of bounded functions with bounded Dirichlet  energies.
Both of them are
of interest in their own. With those two at hand, we give
proofs of
Theorems \ref{thm:mainth1} and \ref{T:1.5}, and the assertion of Example \ref{exa2} in Subsection \ref{S:3.3}.
 In the appendix of this paper,
 in addition to presenting the proofs for Lemma \ref{L:d1} and Proposition \ref{P:1.4},
 we study
the Mosco convergence for $(\E^{\e,\w},\F^{\e,\w})$.

\subsection{Notations}  We use := as a way of definition. Let $\R_+:=[0,\infty)$, $\Z_+:=\{0,1,2,\cdots\}$, $\bS^{d-1}$ be the unit sphere in $\R^d$.  For all $x\in \R^d$ and $r>0$, set $B(x,r)=\{z\in \R^d:|z-x|<r\}$. For $p \in [1,\infty]$ and Lebesgue measurable $A\subset \R^d$,
we use $|A|$ to denote the $d$-dimensional Lebesgue measure  of $A$, $C_b(A)$ the space of bounded and continuous functions on $A$, $L^p(A;dx)$ the space of $L^p$-integrable functions on $A$ with  respect to the Lebesgue measure,
and $L_{loc}^p(\R^d;dx)$ the space of locally $L^p$-integrable functions on $\R^d$ with respect to the Lebesgue measure. Denote $\langle\cdot, \cdot\rangle_{L^2(\R^d;\mu(dx))}$ the inner product in $L^2(\R^d;\mu(dx))$.
Denote by $B(\R^d)$ the set
of locally bounded measurable  functions on $\R^d$, by $B_b(\R^d)$
the set of bounded measurable functions on $\R^d$, and by
$B_c(\R^d)$ the set of bounded measurable functions on $\R^d$ with
compact support.
$C^1_c(\R^d)$ (respectively, $C_c(\R^d)$ or $C_c^\infty(\R^d)$) denotes the space of $C^1$-smooth (respectively, continuous or $C^\infty $-smooth)
functions on $\R^d$ with compact support.

\section{Homogenization of stable-like Dirichlet forms: general results}\label{S:2}

For any $\e>0$, let $\LL^{\e,\w}$ be the generator of the Dirichlet form $(\E^{\e,\w},\F^{\e,\w})$
on $L^2(\R^d; \mu^{\e, \w}(dx))$ given by \eqref{e:1.6}.
Let $\LL^K$ be the generator of the Dirichlet form $(\E^K,\F^K)$
of \eqref{eq:DeDFK} on $L^2(\R^d; dx)$.
The goal of homogenization theory is to construct homogenized characteristics and clarify whether the solutions for the operators $\LL^{\e,\w}$ are close to the solution for the operator $\LL^K$. In this paper, we are concerned
with the following question:
 how to prove that the solution to the equation
\begin{equation}\label{euq-00}
\left(\lambda -\LL^{\e,\w} \right) u^{\e,\w} =f
\end{equation}
 on $L^2(\R^d;\mu^{\e,\w}(dx))$ for any $\lambda>0$ and $f\in C_c(\R^d)$ converges in the resolvent
topology,
as $\e\to0$, to the solution to the equation
\begin{equation}\label{euq-0011}
\left( \lambda -\LL^K \right) u   =f
\end{equation}
on $L^2(\R^d;dx)$.

The section is devoted to addressing this question under the following assumption.

\noindent{\it{\bf Assumption (H):}   There is $\Omega_0\subset \Omega$ of full probability  measure so that the following hold
for every $\w \in \Omega_0$.
\begin{itemize}
\item[(i)] If
 $\{f_\e:\e\in (0,1]\}$ is a sequence of functions on $\R^d$ such that $f_{\e}\in
\FF^{\e,\w}$ for any $\e\in (0,1]$ and
$$
\limsup_{\e \to 0}  \left(\|f_\e\|_\infty+
\EE^{\e,\w}(f_\e,f_\e)\right)<\infty,
$$
then $\{f_\e: \e\in (0,1]\}$ is  pre-compact  as $\e \to 0$  in $L^1(B(0,r);dx)$ for every $r>1$ in the sense that for any sequence
$\{\e_n: n\geq 1\}\subset (0, 1]$ with $\lim_{n\to 0} \e_n =0$, there are a subsequence $\{\e_{n_k}:k\geq 1\}$ and a
function $f \in L^1_{loc}(\R^d;dx)$
 so that  $f_{n_k}$ converges to $f$ in $L^1(B(0,r);dx)$ for every $r>1$.

\item[(ii)] For any $g\in C_c^\infty(\R^d)$,
 \begin{equation}\label{e:2.3}
 \lim_{\eta\to 0}\limsup_{\varepsilon\to 0}
 \iint_{\R^d\times \R^d\setminus \Delta}
 (g(x)- g(y))^2\frac{\kappa ({ x/\e},y/\e;\w)}{|x-y|^{d+\alpha}}\I_{\{|x-y|\le \eta\}}\,dx\,dy=0
 \end{equation}
 and
 \begin{equation}\label{e:2.3-0}
\lim_{\eta\to 0}\limsup_{\varepsilon\to0}
\iint_{\R^d\times \R^d}
(g(x)- g(y))^2\frac{\kappa ({ x/\e},y/\e;\w)}{|x-y|^{d+\alpha}}\I_{\{|x-y|\ge1/\eta\}}\,dx\,dy=0.
 \end{equation}

\item[(iii)] There is a constant $p>1$ such that
 $$
\limsup_{\e\to 0}\int_{B(0, R)}\Big(\int_{B(0, R)} {\kappa ({x/\e},y/\e;\w)}\,dy\Big)^p\,d x<\infty
\quad \hbox{for every } R\geq 1 .
$$

\item[(iv)] For every  $\eta>0$, $f\in B_b(\R^d)$ and $g\in C^\infty_c(\R^d)$,
\begin{align*}
&\lim_{\e\to 0}
\iint_{\R^d\times \R^d}
(f(x)-f(y))(g(x)-g(y))
\frac{\kappa (x/\e,y/\e;\w)}{|x-y|^{d+\alpha}} \I_{\{\eta<|x-y|<1/\eta,  x-y \in \Gamma\}} \,dx\,dy\\
&=
\iint_{\R^d\times \R^d}
(f(x)-f(y))(g(x)-g(y)) \frac{K(x-y) }{|x-y|^{d+\alpha}} \I_{\{\eta<|x-y|<1/\eta,  x-y \in \Gamma\}}  \,dx\,dy ,
\end{align*}
where $K(z)$ is a  measurable even  function on $\R^d$ such that $C_1\le K(z)\le C_2$ for some constants $C_1,C_2>0$.
\end{itemize}}

The section is divided into two parts. We first consider the weak convergence of resolvents in the Dirichlet norm, and then study the strong convergence of resolvents in the $L^2$-norm.

We will use the following Birkhoff ergodic theorem  (see, for example, \cite[Theorem 7.2]{JKO}) several times in this paper.

\begin{proposition}\label{P:2.1}
Suppose that $\nu\geq 0$ is a random variable on $(\Omega, \FF, \bP)$ with $\bE [ \nu ] <\infty$.
There is a subset $\Omega_1\subset \Omega$ of full probability measure so that for every $\w\in \Omega_1$,
the function $x\mapsto \nu (\tau_{x/\e} \w )$ converges weakly to $\bE [ \nu]$  in $L^1_{loc} (\R^d;dx)$ as $\e \to 0$;
that is, for every  $\w \in \Omega_1$, every bounded Lebesgue measurable set $K\subset \R^d$
and every  $\varphi \in L^\infty (K; dx)$,
$$
 \lim_{\e \to 0} \int_{K} \varphi (x) \nu (\tau_{x/\e} \w )\,dx
 =  \bE [\nu]\int_{K} \varphi (x)  \, dx.$$
Furthermore, if $\bE [ \nu^p]<\infty $ for some $p>1$, then for every $\w \in \Omega_1$, the function
$x\mapsto \nu (\tau_{x/\e} \w )$ converges weakly to $\bE [ \nu ]$ in $L^p_{loc} (\R^d;dx)$ as $\e \to 0$.
\end{proposition}

\subsection{Weak convergence of resolvents}

Recall that $\mu$ is a random variable on $(\Omega, \FF, \bP)$ so that $\bE [ \mu ] =1$ and for any $\w \in \Omega$,
$\mu (\tau_x \w ) >0$ for a.e.\ $x\in \R^d$.
Denote by $U_\lambda^{\e,\w}$   the $\lambda$-order resolvent of the regular
Dirichlet form $(\E^{\e,\w},\F^{\e,\w})$ on $L^2(\R^d; \mu^{\e, \w} (dx))$, and $U_\lambda^K$  the
$\lambda$-order resolvent of the  regular Dirichlet form $(\E^K,\F^K)$ on $L^2(\R^d;  dx)$.
It is well known that $U_\lambda^{\e,\w}f $  and $U_\lambda^Kf $  are the unique solution to \eqref{euq-00} and \eqref{euq-0011},
respectively.

\begin{theorem}\label{T:2.2}
Suppose that assumption {\bf (H)} holds and that $\bE [\mu^p]<\infty$ for some $p>1$.  Then
there is a subset $\Omega_2\subset \Omega$ of full probability measure
so that  for every  $\w\in \Omega_2$ and any $f\in C_c(\R^d)$,
\begin{equation}\label{e:2.4}
U_\lambda^{\e,\w} f  \hbox{ converges to }  U_\lambda^Kf \hbox{ in  } L^1(B(0,r);dx)\quad \hbox{ as  }\e \to 0
\end{equation}
for every $r>1$,
\begin{equation}\label{e:2.5}
\lim_{\e \to 0} \langle U_\lambda^{\varepsilon } f,g\rangle_{L^2(\R^d; \mu^{\e}(dx))}=
\< U^K_\lambda f, g\>_{L^2(\R^d; dx)} \quad \hbox{for every } g\in C_c(\R^d),
\end{equation}
and
\begin{equation}\label{euq-0022}
\lim_{\e\to0} \EE^{\e,\w}(U_\lambda^{\e,\w} f,g)={\EE}^K(U_\lambda^K f, g)
\quad \hbox{for every } g\in C_c^\infty (\R^d),
\end{equation}
where $K(z)$ is the function in assumption {\bf (H)}{\rm(iv)}.
\end{theorem}

\begin{proof}
Let $\Omega_2=\Omega_0 \cap \Omega_1$, where $\Omega_0$ and $\Omega_1$ are
the subset of $\Omega$ in assumption {\bf (H)} and  in Proposition \ref{P:2.1}, respectively,
both of them having full probability measure.
Let $\w \in \Omega_2$.
For simplicity, throughout the proof we omit the parameter $\w$
from $U_\lambda^{\e,\w}$ and $\mu^{\e,\w}$.

Fix $\lambda>0$.
For any $\varepsilon>0$ and $f \in C_c  (\R^d)$, $U_\lambda^{\varepsilon}f$ is the unique element in $ \F^{\e,\w}$
so that
 \begin{equation}\label{e:sy2}
\lambda\langle U_\lambda^{\varepsilon} f,g\rangle_{L^2(\R^d;
\mu^{\e}(dx))}+\EE^{\varepsilon}(U_\lambda^{\varepsilon}f,g)=
\langle f, g\rangle_{L^2(\R^d; \mu^{\e}(dx))} \quad \hbox{for } g\in \F^{\e,\w}.
\end{equation}

We first treat the limits for the right hand side and the first term in the left hand side in \eqref{e:sy2}.  Note that
$\|U_\lambda^\e f\|_{L^2(\R^d; \mu^\e(dx))} \le \lambda^{-1}\|f\|_{L^2(\R^d;\mu^\e(dx))}$ for $\e\in (0,1]$.
Thus by \eqref{e:sy2},
\begin{equation}\label{t2-1-4a}
\lambda \|U_\lambda^\e f\|_{L^2(\R^d;\mu^\e(dx))}^2+
\EE^{\varepsilon}(U_\lambda^{\varepsilon}f,U_\lambda^\e f)
= \langle f, U_\lambda^\e f\rangle_{L^2(\R^d; \mu^{\e}(dx))} \leq
  \lambda^{-1}\|f\|_{L^2(\R^d;\mu^\e(dx))}^2.
\end{equation}
Hence, according to $\Ee[\mu]=1$ and Proposition \ref{P:2.1},
$$
\limsup_{\e \to 0}
\big(\EE^{\varepsilon}(U_\lambda^{\varepsilon}f,U_\lambda^\e
f)+ \|U_\lambda^\e f\|_{L^2(\R^d;\mu^\e(dx))}^2\big)<\infty.
$$
Note also that $\|U_\lambda^\e f\|_\infty\le \lambda^{-1}\|f\|_\infty$ for all $\varepsilon\in (0,1]$.
By assumption {\bf (H)}(i),
$\{U_\lambda^\e f: \e\in (0,1]\}$   is
pre-compact  as $\e \to 0$  in $L^1(B(0,r);dx)$ for every $r>1$. In particular,
for any  sequence $\{\e_n: n\ge 1\}\subset (0, 1]$ with $\lim_{n\to \infty} \e_n=0$,
we can find a subsequence
$\{\e_{n_k}: k\geq 1\}$ (for simplicity we will still denote it by $\{\e_n: n\geq 1\}$)
and a
function $U_\lambda^* f\in L^1_{loc}(\R^d;dx)$ (which indeed may depend on $\w$) so that
\begin{equation}\label{e:2.6}
\lim_{n \rightarrow \infty}\int_{B(0,r)}|U_\lambda^{\e_n}
f(x)-U_\lambda^* f(x)|\,dx=0 \quad \hbox{for every } r>1.
\end{equation}
Since $\sup_{\e\in (0,1]}\|U_\lambda ^\e f\|_\infty\le \lambda^{-1}\|f\|_\infty$,
we in fact have $\|U_\lambda ^* f\|_\infty\le \lambda^{-1}\|f\|_\infty$ and so for every $q\in [1, \infty)$,
\begin{equation}\label{t2-1-6}
\lim_{n \rightarrow \infty}\int_{B(0,r)}|U_\lambda^{\e_n}
f(x)-U_\lambda^* f(x)|^q\, dx=0\quad \hbox{for every } r>1.
\end{equation}
We next show that $U^*_\lambda f= U^K_\lambda f$, where $U^K_\lambda$ is the
$\lambda$-order resolvent associated with the operator $\LL^K$.

Fix $g\in C_c(\R^d)$. We choose $R>0$ such that $\text{supp}[g]\subset B(0,R_0)$. Then, by the H\"older inequality with $p>1$ from assumption {\bf (H)}(iii),
\begin{align*}
&\int_{\R^d} |U_\lambda^{\e_n} f(x)-U^*_\lambda f(x)|g(x)\mu
(\tau_{x/\e_n } \w)  \,dx\\
&\le \|g\|_\infty \left(\int_{B(0,R_0)}|U_\lambda^{\e_n}
f(x)-U^*_\lambda f(x)|^{{p}/({p-1})}\,dx\right)^{(p-1)/p}
\left(\int_{B(0,R_0)}\mu
(\tau_{x/\e_n } \w)^{p} \,dx\right)^{1/p}.
\end{align*}
Since $\Ee[\mu^p]<\infty$, we get by Proposition \ref{P:2.1} that
for $\w \in \Omega_2 \subset \Omega_1$,
$$
\limsup_{n\to\infty}\int_{B(0,R_0)}\mu^p (\tau_{x/\e_n} \w )\,dx<\infty.
$$
This along with \eqref{t2-1-6} yields
$$
\lim_{n \to \infty}\int_{\R^d} |U_\lambda^{\e_n} f(x)-U^*_\lambda
f(x)|g(x)
\mu (\tau_{x/\e_n } \w ) \,dx=0.
$$
On the other hand,
since $\Ee [ \mu]=1$, we have again by Proposition \ref{P:2.1},
\begin{align*}
\lim_{n \to \infty}\int_{\R^d}U^*_\lambda f(x)
g(x)\mu(\tau_{x/\e_n}\w)\,dx& =\lim_{n \to
\infty}\int_{B(0,R_0)}U^*_\lambda f(x) g(x)
\mu (\tau_{x/\e_n} \w) \, dx \\
&=
\int_{B(0,R_0)}U^*_\lambda f(x) g(x)\Ee
\left[ \mu\right]
\,dx=
\int_{\R^d}U^*_\lambda f(x) g(x)\,dx.
\end{align*}
Putting both estimates above together yields that for any $f,g\in C_c(\R^d)$
\begin{equation}\label{t2-1-5a}
\lim_{n \to \infty}\lambda\langle U_\lambda^{\varepsilon_n} f,g\rangle_{L^2(\R^d; \mu^{\e_n}(dx))}=
\lambda\langle U^*_\lambda f,g\rangle_{L^2(\R^d;dx)}.
\end{equation}

Similarly, according to $\Ee [ \mu]=1$ and Proposition \ref{P:2.1}, we have for every $f, g\in C_c(\R^d)$,
\begin{equation}\label{t2-1-5b}
\lim_{n \to \infty}\langle f,g\rangle_{L^2(\R^d;\mu^{\e_n}(dx))}=
\langle f,g\rangle_{L^2(\R^d; dx)}.
\end{equation}
In particular, by \eqref{t2-1-5a} and \eqref{t2-1-5b},
we obtain
that for every $g\in C_c(\R^d)$,
\begin{align*}
\langle U_\lambda^* f, g\rangle_{L^2(\R^d;dx)}&
=\lim_{n \to \infty}\langle U_\lambda^{\e_n}f, g\rangle_{L^2(\R^d;\mu^{\e_n}(dx))}\\
&\le \lim_{n \to \infty}(\|U_\lambda^{\e_n}f\|_{L^2(\R^d;\mu^{\e_n}(dx))}\cdot \|g\|_{L^2(\R^d;\mu^{\e_n}(dx))})\\
&\le \lambda^{-1}\lim_{n \to \infty}(\|f\|_{L^2(\R^d;\mu^{\e_n}(dx))}\cdot \|g\|_{L^2(\R^d;\mu^{\e_n}(dx))})\\
&=\lambda^{-1}\|f\|_{L^2(\R^d;dx)}\cdot \|g\|_{L^2(\R^d;dx)},
\end{align*}
which implies immediately that $U_\lambda^*f \in L^2(\R^d;dx)$ and
$\|U_\lambda^* f\|_{L^2(\R^d;dx)}\le \lambda^{-1}\|f\|_{L^2(\R^d;dx)}$.

We now treat the second term in the left hand side of \eqref{e:sy2}
with $g\in C^\infty_c(\R^d)$.
According to Lemma \ref{L:d1}, it holds that for any $0<\eta\le 1$,
\begin{align*}
 2\EE^{\varepsilon_n}(U_\lambda^{\varepsilon_n}f,g)
&=\iint_{\R^d\times \R^d \setminus \Delta} (U_\lambda^{\varepsilon_n}f(x+z)- U_\lambda^{\varepsilon_n}f(x))(g(x+z)-g(x))\frac{\kappa (0,z/\varepsilon_n ;\tau_{x/\varepsilon_n }\w)}{|z|^{d+\alpha}} \,
\I_{\Gamma\cap\{|z|\leq \eta\}}  \,dz dx \\
&\quad+\iint_{\R^d\times \R^d} (U_\lambda^{\varepsilon_n}f(x+z)- U_\lambda^{\varepsilon_n}f(x))(g(x+z)-g(x))\frac{\kappa (0,z/\varepsilon_n ;\tau_{x/\varepsilon_n }\w)}{|z|^{d+\alpha}}
\, \I_{\Gamma \cap \{|z|\ge 1/\eta\}} \,dz dx\\
&\quad-2\langle U_\lambda^{\varepsilon_n}f, \LL^{n}_{\eta} g\rangle_{L^2(\R^d; dx)}\\
&=:I_1^{n,\eta}+I_2^{n,\eta}-2I_{3}^{n,\eta},
\end{align*}
where
$$
\LL_{\eta}^{n} f(x)=\int_{\R^d} (f(x+z)-f(x))\frac{\kappa (0,z/\varepsilon_n ;\tau_{x/\varepsilon_n }\w)}{|z|^{d+\alpha}}\, \I_{\Gamma \cap \{\eta<|z|< 1/\eta\}} \,dz.
$$
Note that
$$
\LL_{0}^{n}f(x):={\rm p.v.} \int_{\R^d}
(f(x+z)-f(x))\frac{\kappa (0,z/\varepsilon_n ;\tau_{x/\varepsilon_n}\w)}{|z|^{d+\alpha}}
\, \I_{\Gamma}(z) \,dz
$$
typically is the generator of the Dirichlet form $(\E^{\e_n, \w}, \F^{\e_n,\w})$ on $L^2(\R^d;dx)$.

By the Cauchy-Schwarz inequality,
$$
I_1^{n,\eta}
\leq  \sqrt{2\EE^{\varepsilon_n}(U_\lambda^{\varepsilon_n}f, U_\lambda^{\varepsilon_n}f)}
\sqrt{\iint_{\R^d\times \R^d\setminus \Delta} (g(x+z)- g(x))^2\frac{\kappa (0,z/\varepsilon_n ;\tau_{x/\varepsilon_n }\w)}{|z|^{d+\alpha}}\, \I_{\{|z|\leq \eta\}} \,dz\,dx}.
$$
 This along with \eqref{t2-1-4a} and assumption {\bf (H)}(ii) yield that
$$
\lim_{\eta \to 0} \lim_{n\to \infty}I_1^{n,\eta}=0.
$$
Similarly, we can verify
$$
\lim_{\eta \to 0}\lim_{n\to \infty}I_2^{n,\eta}=0.
$$

Set
$$ \LL^K_{\eta}g(x):=\int_{\R^d}(g(x+z)-g(x))\frac{ K(z)}{|z|^{d+\alpha}}\I_{\Gamma \cap \{\eta<|z|< 1/\eta\}}\,dz.$$
It is obvious that
$$\lim_{\eta\to0 }  \LL^K_{\eta }g(x)= \LL^K g(x),\quad g\in C_c^\infty(\R^d),$$
where
$$
\LL^K g(x):=\text{p.v.}\int_{\R^d}\left(g(x+z)-g(x)\right)\frac{K(z)}{|z|^{d+\alpha}}
\, \I_{\Gamma}(z) \,dz
$$
is the infinitesimal generator with respect to
$(\EE^K,\FF^K)$ on $L^2(\R^d;dx)$. We find that
\begin{align*}
|I_3^{n,\eta}-\langle U^*_\lambda f, \LL^K_{\eta}
g\rangle_{L^2(\R^d;dx)}|
&\le |\langle U_\lambda^{\varepsilon_n}f-  U^*_\lambda f, \LL_{\eta}^{n}g\rangle_{L^2(\R^d;dx)}|
 + |\langle U^*_\lambda f,
 (\LL_{\eta}^{n} g - \LL^K_{\eta} g)\rangle_{L^2(\R^d;dx)}|\\
&=:I_{3,1}^{n,\eta}+ I_{3,2}^{n, \eta}.
\end{align*}
Observe that, for every $x\in \R^d$,
$$
|\LL_\eta^ng(x)|\le 2\|g\|_\infty \eta^{-d-\alpha}\left(\int_{B(0,R_0+1/\eta)}\kappa(x/\e_n,y/\e_n;\w)\,dy\right) \,\I_{\{B(0,R_0+1/\eta)\}}(x).
$$
This along with the H\"older inequality with $p>1$ in assumption {\bf (H)}(iii) yields that
\begin{align*}
I_{3,1}^{n,\eta}\le & c_{11}(\eta,\w) \left(\int_{B(0,R_0+1/\eta)}
|U_\lambda^{\e_n} f(x)-U^*_\lambda f(x)|^{{p}/({p-1})}\,dx\right)^{1-1/p}
\\
&\times\Big(\int_{B(0,R_0+1/\eta)}
\bigg(\int_{B(0,R_0+1/\eta)}\kappa(x/\e_n,y/\e_n;\w)\,dy\Big)^p\,dx\bigg)^{1/p}.
\end{align*}
 Hence, by
\eqref{t2-1-6} and assumption {\bf (H)}(iii), for each fixed $\eta>0$,
$\lim_{n \to \infty}I_{3,1}^{n,\eta}=0$.
On the other hand,
\begin{align*}
&\int_{\R^d} U^*_\lambda f(x) \LL_{\eta}^{n}g(x)\,dx\\
&=\!-\frac{1}{2}\!\iint_{\R^d\times \R^d} \left(U^*_\lambda
f(x+z)\!-U^*_\lambda
f(x)\right) \left(g(x+z)\!-g(x)\right)
\frac{ \kappa (x/\e_n ,(x+z)/\e_n;\w)}{|z|^{d+\alpha}}\, \I_{\Gamma \cap\{\eta<|z|<1/\eta\}} \, dz\,dx.
\end{align*}
Then,  by applying assumption {\bf (H)}(iv), we have
\begin{align*}
&\lim_{n \to \infty}\int_{\R^d} U^*_\lambda f(x) \LL_{\eta}^{n}g(x)\,dx\\
&=-\frac{1}{2}\iint_{\R^d\times \R^d}
(U^*_\lambda f(x+z)-U^*_\lambda f(x)) (g(x+z)-g(x))
\frac{ K(z)}{|z|^{d+\alpha}}\, \I_{\Gamma \cap\{\eta<|z|<1/\eta\}}\,dz\,dx\\
&=\int_{\R^d} U^*_\lambda f(x)  \LL^K_\eta g(x)\,dx,
\end{align*}
which implies $\lim_{n \to \infty}I_{3,2}^{n,\eta}=0$.

Combining all the estimates for $I_i^{n, \eta}$, $i=1,2,3$ with
the fact that $U_\lambda^* f\in L^2(\R^d;dx)$,
first letting $n \to \infty$ and then letting $\eta \to 0$,
we obtain
\begin{equation}\label{e:oopp}
\lim_{n \to \infty}\EE^{\e_n}\big(U_\lambda^{\e_n} f,g\big)=
-\int_{\R^d} U^*_\lambda f(x)  \LL^K g(x)\,dx.
\end{equation}
Putting this with \eqref{e:sy2}, \eqref{t2-1-5a} and \eqref{t2-1-5b} together,
we see
that for any
$f\in C_c(\R^d)$ and $g\in C_c^\infty (\R^d)$,
$$
\langle (\lambda- \LL^K)g, U^*_\lambda f\rangle_{L^2(\R^d;dx)}=\langle
f, g\rangle_{L^2(\R^d; dx)}.
$$
In particular,
since the above holds for any $g\in C^\infty_c(\R^d)$, we have
$U^*_\lambda
f=U^K_\lambda f$; that is, $U^*_\lambda$ is   the
$\lambda$-order resolvent corresponding to the operator $ \LL^K$.
For every $f\in C_c(\R^d)$ and $g\in C_c^\infty (\R^d)$,
$$
\lim_{n \to \infty}\EE^{\e_n}\big(U_\lambda^{\e_n} f,g\big)={{\EE}}^K(U^K_\lambda f, g),
$$
and by \eqref{e:2.6},  $U_\lambda^{\e_n} f$ converges to $U^K_\lambda$ in $L^1(B(0,r); dx)$, as $n\to \infty$, for every $r>1$.
Since these hold for any sequence  $\{\e_n\}_{n\ge1}$ that converges to $0$, we get \eqref{euq-0022} and \eqref{e:2.4}.
The property \eqref{e:2.5} follows from \eqref{t2-1-5a}.
\end{proof}

\subsection{Strong convergence of resolvents}

\begin{theorem}\label{C:note3}
Suppose that assumption {\bf (H)} holds and $\Ee[\mu^p]<\infty$ for some $p>1$. Let $\Omega_2$ be the subset of $\Omega$ in Theorem $\ref{T:2.2}$ that is of full probability measure.
Then for every  $\w\in \Omega_2$ and every  $f\in C_c (\R^d)$,
\begin{equation}\label{e:002}
\lim_{\e \to 0}\|U_\lambda^K f\|_{L^2(\R^d;\mu^{\e,\w}(dx))}=\|U_\lambda^K f\|_{L^2(\R^d;dx)}
\end{equation} and
\begin{equation}\label{t2-1-2}
\lim_{\e \to 0}\|U_\lambda^{\e,\w}f-U_\lambda^Kf \|_{L^2(\R^d;\mu^{\e,\w}(dx))}=0.
\end{equation}
In particular,
\begin{equation}\label{t2-1-2-11}
\lim_{\e \to 0}\|U^{\e,\w}_\lambda f\|_{L^2(\R^d;\mu^{\e,\w}(dx))}= \|U_\lambda^K f\|_{L^2(\R^d; dx)}.
\end{equation}
\end{theorem}

To prove Theorem $\ref{C:note3}$, we need the following lemma.
Recall that
$\Gamma$ is an infinite symmetric cone in $\R^d$ that has non-empty interior.
We define the Dirichlet form $(\EE_0,\FF_0)$
on $L^2(\R^d; dx)$ as follows
\begin{equation}\label{e5-1}
 \EE_0(f,f)=
 \iint_{\R^d \times \R^d \setminus \Delta}
\frac{(f(y)-f(x))^2}{|y-x|^{d+\alpha}} \, \I_{\Gamma}(x-y)\, dx\,dy,\quad f\in \FF_0,
\end{equation}
where $\FF_0$ is the closure of $C_c^\infty(\R^d)$ under the norm
$\big(\EE_0(\cdot,\cdot)+\|\cdot\|^2_{L^2(\R^d;dx)}\big)^{1/2}$.

\begin{lemma}\label{l5-1}There exists a constant $c_0>0$ such that for all
$f\in C^1_c(\R^d)$,
$$
\|f\|^2_{L^2(\R^d;dx)}\le c_0\EE_0(f,f)^{{d}/({d+\alpha})}\|f\|_{L^1(\R^d;dx)}^{{2\alpha}/({d+\alpha})}.
$$
\end{lemma}

\begin{proof} This inequality is well-known when $\Gamma=\R^d$ (see, for instance, \cite[Proposition 3.1]{CK}).
By \cite[Theorem 1.1]{BKS}, there is a constant $c_1>0$ which may depend on $\Gamma$ so that
for every $f\in L^2 (\R^d; dx)$,
\begin{equation}\label{e:2.18}
 \iint_{\R^d \times \R^d \setminus \Delta}
\frac{(f(y)-f(x))^2}{|y-x|^{d+\alpha}}  \, dx\,dy\le
c_1
 \iint_{\R^d \times \R^d \setminus \Delta}
\frac{(f(y)-f(x))^2}{|y-x|^{d+\alpha}} \, \I_{\Gamma}(x-y)\, dx\,dy.
 \end{equation}
This immediately gives the desired result.  In fact, the result of \cite[Theorem 1.1]{BKS} is more general which shows that \eqref{e:2.18}
holds with $\R^d $ being replaced by any ball $B$ in $\R^d$ on both sides of \eqref{e:2.18}. In the
$\R^d$ case, one can establish \eqref{e:2.18} directly by using the Fourier transform. For reader's convenience, we give a such proof below.

For $f\in C^1_c (\R^d)$,
let
$$
\hat f(\xi):=(2\pi)^{-d/2}\int_{\R^d} e^{i\langle \xi,y\rangle} f(y) \,dy
$$
 be
the Fourier transform of $f$. Then,
$$
\int_{\R^d} f(x)^2\,dx=\int_{\R^d}|\hat f(\xi)|^2\, d\xi,\quad \EE_0(f,f)=\int_{\R^d}|\hat f(\xi)|^2\phi(\xi)\,d\xi,
$$
where
\begin{equation}\label{e:2.19}
\phi(\xi)=\int_{\Gamma}\frac{1-e^{i\langle \xi, z\rangle}}{|z|^{d+\alpha}}\,dz
=\int_{\Gamma}\frac{1-\cos\langle \xi, z\rangle}{|z|^{d+\alpha}}\,dz.
\end{equation}
Let $\bS^{d-1}$ be the unit sphere in $\R^d$ and $\Theta = \Gamma \cap \bS^{d-1}$. Denote by $z=(r, \theta)$ the spherical coordinates for $0\neq z\in\R^d$, and $\sigma$ the Lebesgue surface measure on $\bS^{d-1}$.
Since $\Theta$
has non-empty interior,
there are positive constants $\delta_1$ and $\delta_2$ such that
for any $\eta\in \Ss^{d-1}$, there exists
$A(\eta)  \subset \Theta$ (which may depend on $\eta$) so that
$\sigma (A(\eta)) \geq \delta_1$ and
$$
1-\cos(r\langle \eta,\theta_0\rangle)\ge \delta_2
\quad \hbox{ for all } 1/2\le r\le 1 \hbox{ and } \theta_0\in A(\eta).
$$
Thus for any $\xi\in \R^d$,
\begin{equation}\begin{split}\label{e:2.20--}
\phi(\xi)& = c_2|\xi|^\alpha\int_0^\infty r^{-1-\alpha}\int_{\Theta}\big(1-\cos(r\langle {\xi}/{|\xi|},\theta\rangle)\big)\,\sigma (d\theta)\,dr  \\
&\ge  c_3|\xi|^\alpha\int_{{1}/{2}}^1  \int_{ A(\xi/|\xi|)}\big(1-\cos(r\langle {\xi}/{|\xi|},\theta\rangle)\big)\,d\theta\, dr
\ge \frac{c_3\delta_1\delta_2}{2}|\xi|^\alpha.
\end{split}\end{equation}
Therefore, for every $f\in C^1_c (\R^d)$,
$$
\EE_0(f,f)\geq c_4 \int_{\R^d}|\hat f(\xi)|^2|\xi |^\alpha \,d\xi
=c_4  \iint_{\R^d \times \R^d \setminus \Delta}
\frac{(f(y)-f(x))^2}{|y-x|^{d+\alpha}}  \, dx\,dy.
$$
This establishes \eqref{e:2.18}.
\end{proof}

 \medskip

\begin{proof}[Proof of Theorem $\ref{C:note3}$]
Throughout this proof,  we again suppress  the parameter $\w$ for simplicity.
We claim  that \eqref{e:002} holds for any $f\in C_c (\R^d)$ and $\w\in \Omega_2$.
Indeed, let $X:=\{X_t, t\ge 0\}$ be the symmetric   L\'evy
 process associated with the Dirichlet form $(\EE^K,\FF^K)$ on $L^2(\R^d; dx)$.
 The L\'evy process $X$ has L\'evy exponent
 $$
 \psi (\xi)  = \int_{\Gamma} \left( 1-\cos\langle \xi, z\rangle \right)\frac{K(z) }{|z|^{d+\alpha}}\,dz,
 \quad \xi \in \R^d.
 $$
 Since $0<C_1\leq K(z) \leq C_2$, we have by \eqref{e:2.19} and \eqref{e:2.20--} that $\psi (\xi) \geq c_0 |\xi|^\alpha$ on $\R^d$.
 Hence the L\'evy process $X$ has a jointly continuous transition density function $p(t, x, y)=p(t, x-y)$ with respect to the Lebesgue
 measure on $\R^d$, where
 $$
 p(t, x) := (2\pi)^{-d/2}\int_{\R^d}e^{-i\langle \xi, x\rangle} e^{-t \psi (\xi)} \,d\xi , \quad x\in \R^d.
 $$
Since
$$
\E_0(g ,  g)\le C_1^{-1}\E^K(g, g) \quad \hbox{for any } g\in C_c^1(\R^d),
$$
we have by Lemma \ref{l5-1} that
$$
\|g\|^2_{L^2(\R^d;dx)}\le c_1\big(\EE^K(g, g)\big)^{{d}/({d+\alpha})}\|g\|_{L^1(\R^d;dx)}^{{2\alpha}/({d+\alpha})},
\quad g\in C_c^1 (\R^d).
$$
This along with  \cite[Theorems (2.1) and (2.9)]{CKS} yields
$$ p(t,x-y)\le c_2t^{-d/\alpha} \quad \hbox{for every } t>0
  \hbox{ and }  x,y\in \R^d.
 $$
Since $K(z)\le C_2$ on $\R^d$, by the proof of \cite[Theorem 1.4]{BGK}, we in fact have
$$
p(t,x-y)\le c_3\left(t^{-d/\alpha}\wedge \frac{t}{|x-y|^{d+\alpha}}\right) \quad \hbox{for } t>0
\hbox{ and }  x,y\in \R^d.
$$

In the following,
we fix $f\in C_c  (\R^d)$. Choose $R_0\geq 1 $ so that $\text{supp} [ f]\subset
B(0,R_0)$. For any $x\in \R^d$ with $|x|>2R_0$,
\begin{equation}\label{e:2.21}\begin{split}
U_\lambda^K f(x)&=  \int_0^\infty\int_{\R^d} e^{-\lambda t}p(t,x-y)f(y)\,dy\,dt  \\
&\le  c_4\|f\|_\infty \int_0^\infty \int_{B(0,R_0)} e^{-\lambda t}\left(t^{-d/\alpha}\wedge \frac{t}{|x-y|^{d+\alpha}}\right)\,dy\,dt
\le  c_5\, |x|^{-d-\alpha}.
\end{split}\end{equation}
where $c_5>0$ is a constant that depends on $\|f\|_\infty$, $\lambda$ and $R_0$.

As mentioned in the beginning of Theorem \ref{T:2.2}, $\Omega_2=\Omega_0 \cap \Omega_1$, where $\Omega_0$ and $\Omega_1$ are
the subset of $\Omega$ in assumption {\bf (H)} and  in Proposition \ref{P:2.1}, respectively.
According to Proposition \ref{P:2.1},
for every
$\w \in \Omega_2$, $k\ge 0$ and every $\varphi \in L^\infty (B(0,2^k); dx)$
\begin{equation}\label{e:2.20}
\lim_{\e \to 0} \int_{B(0,2^k)}\varphi(x)\mu (\tau_{x/\e} \w) \,dx =
\bE [\mu ] \int_{B(0,2^k)}  \varphi (x) \,dx=\int_{B(0,2^k)}  \varphi (x) \,dx.
\end{equation}
In particular, taking $\varphi\equiv 1$ and $k=0$ yields that
\begin{equation*}
\lim_{\e \to 0}\e^{d}\int_{B(0,1/\e)}\mu(\tau_x\w)\,dx
=\lim_{\e \to 0} \int_{B(0,1)}
\mu (\tau_{x/\e} \w) \,dx  =|B(0,1)|.
\end{equation*}
Hence, for every $\w\in \Omega_2$, there exists $\e_0(\w)>0$ such that
\begin{equation}\label{p4-2-1a}
\int_{B(0,1/\e)}\mu(\tau_x\w)\,dx\le 2\e^{-d}|B(0,1)|\quad \hbox{for all } \e\in (0,\e_0(\w)).
\end{equation}
Let $Q_k:=\left\{x\in \R^d:  2^{k-1} \leq |x| <2^k \right\}$ for $k\geq 1$. Therefore, by \eqref{e:2.21} and
\eqref{p4-2-1a}, we obtain for every $\w\in \Omega_2$ and $N\ge 1$ with $2^N>\max\{2R_0,\e_0(\w)^{-1}\}$,
\begin{equation}\label{p4-2-1}\begin{split}
  \limsup_{\e \to 0}
 \int_{\{|x|\ge 2^N\}}U_\lambda^K f(x)^2\mu (\tau_{x/\e} \w )\,dx
 &=\limsup_{\e \to 0} \sum_{k=N+1}^\infty  \int_{Q_k}U_\lambda^K f(x)^2 \mu (\tau_{x/\e} \w )\,dx   \\
 &\leq c_5^2  \, \limsup_{\e \to 0}  \sum_{k=N+1}^\infty  2^{-2(k-1)(d+\alpha)}
  \int_{B(0,2^{k+1})}
  \mu (\tau_{x/\e} \w )\,dx    \\
  &=c_5^2 \limsup_{\e \to 0}  \sum_{k=N+1}^\infty  2^{-2(k-1)(d+\alpha)}
  \e^d\int_{B(0,2^{k+1}/\e)}
  \mu (\tau_{x} \w )\,dx\\
  &\leq c_6 \sum_{k=N}^\infty    2^{-  2k(d+\alpha) }\e^d\left(2^{k+1}/\e\right)^d
 \le c_7 2^{-(d+2\alpha) N}.
 \end{split}\end{equation}
On the other hand, according to \eqref{e:2.20}, for every $\w \in \Omega_2$,
$$
\lim_{\e\to 0}\int U_\lambda^K f(x)^2\I_{\{|x| < 2^N\}}\mu(\tau_{x/\e } \w)\,dx=
\int U_\lambda^K f(x)^2\I_{\{|x| < 2^N\}}\,dx\quad \hbox{for every } N\ge1.
$$
 This together with \eqref{p4-2-1} and by taking $N\to \infty$ gives  us \eqref{e:002}.

For $\lambda >0$, define
$$
\EE^\e_\lambda  (u, v)= \EE^\e (u, v) + \lambda \<u, v \>_{L^2(\R^d; \mu^\e (dx))}  \quad \hbox{for } u, v\in \FF^\e,
$$
and
$$
 \EE^K_\lambda  (u, v)= \EE^K (u, v) + \lambda \<u, v\>_{L^2(\R^d;dx)}  \quad \hbox{for } u, v \in \FF^K.
$$
For $f\in C_c(\R^d)$ and $g\in C_c^\infty (\R^d)$,
\begin{equation}\label{p4-2-2}
\begin{split}
\< U_\lambda ^\e f, f\>_{L^2(\R^d; \mu^\e (dx))}
&=  \EE_\lambda^{\e}\big(U_\lambda^{\e}f,U_\lambda^{\e}f\big) \\
&= \EE_\lambda^{\e}\big(U_\lambda^{\e}f-g,U_\lambda^{\e}f-g\big)+
\EE_\lambda^{\e} \big(g,g\big)+2\EE_\lambda^{\e} \big(U_\lambda^{\e}f-g,g\big)\\
&= \EE_\lambda^{\e}\big(U_\lambda^{\e}f-g,U_\lambda^{\e}f-g\big)
 +2\< f,g\>_{L^2(\R^d; \mu^\e(dx))} -\EE^{\e}_\lambda \big(g,g\big).
\end{split}
\end{equation}
By (ii) and (iv) of assumption {\bf (H)} and Proposition \ref{P:2.1},
 $$
 \lim_{\e \to 0} \EE^{\e}_\lambda  (g,g) =\EE^K_\lambda(g,g)
\quad \hbox{and} \quad  \lim_{\e \to 0}
\< f,g\>_{L^2(\R^d; \mu^\e(dx))} = \< f, g\>_{L^2(\R^d; dx)}.
$$
Hence after taking $\e \to 0$ in \eqref{p4-2-2}, we get together with \eqref{e:2.5} that
\begin{equation}\label{e:2.22}
\lim_{\e \to 0} \EE_\lambda^{\e}\big(U_\lambda^{\e}f-g,U_\lambda^{\e}f-g\big)
= -2 \< f, g\>_{L^2(\R^d; dx)} + \EE^K_\lambda(g,g) + \< U^K_\lambda f, f\>_{L^2(\R^d; dx)}.
\end{equation}
Take $\{g_k; k\geq 1\}\subset C^\infty_c(\R^d)$ so that $g_k \to U^K_\lambda f$ in $\sqrt{\EE^K_1}$-norm.
Then
\begin{align*}
& \lim_{k\to \infty}
\left( 2 \< f, g_k\>_{L^2(\R^d; dx)} - \EE^K_\lambda(g_k,g_k) - \< U^K_\lambda f, f\>_{L^2(\R^d; dx)}\right) \\
&=2 \< f, U^K_\lambda f\>_{L^2(\R^d; dx)} - \EE^K_\lambda(U^K_\lambda f, U^K_\lambda f) - \< U^K_\lambda f , f\>_{L^2(\R^d; dx)} = 0.
\end{align*}
In particular, we have by \eqref{e:2.22} that
\begin{equation}\label{e:2.23}
\lim_{k\to \infty} \lim_{\e \to 0} \| U_\lambda^{\e}f-g_k\|_{L^2(\R^d; \mu^\e (dx))}=0.
\end{equation}
On the other hand,  by
Proposition \ref{P:2.1},
\eqref{e:2.5}
 and  \eqref{e:002},
\begin{align*}
\lim_{\e \to 0} \| g_k- U_\lambda^Kf \|^2_{L^2(\R^d; \mu^\e (dx))}
&= \lim_{\e \to 0} \left(  \| g_k\|^2_{L^2(\R^d; \mu^\e (dx))}- 2 \<U^K_\lambda f, g_k\>_{L^2(\R^d; \mu^\e (dx))}
         + \| U_\lambda^Kf \|^2_{L^2(\R^d; \mu^\e (dx))} \right)\\
&= \| g_k\|^2_{L^2(\R^d; dx)}- 2 \<U^K_\lambda f, g_k\>_{L^2(\R^d; dx)}
          +  \| U_\lambda^Kf \|^2_{L^2(\R^d; dx)} .
\end{align*}
Hence we have
$$
\lim_{k\to \infty} \lim_{\e \to 0} \|g_k -  U_\lambda^Kf \|^2_{L^2(\R^d; \mu^\e (dx))} =0.
$$
This together with \eqref{e:2.23} gives  \eqref{t2-1-2}.

Consequently, we have by  \eqref{e:002} that
$$
\lim_{\e \to 0}  \| U_\lambda^\e f  \|_{L^2(\R^d; \mu^\e (dx))}
=    \lim_{\e \to 0}  \| U_\lambda^K f \|_{L^2(\R^d; \mu^\e (dx))}
= \| U^K_\lambda f \|_{L^2(\R^d; dx)}.
$$
This completes the proof of the theorem.
\end{proof}

The proof for the property \eqref{t2-1-2} (in particular see \eqref{e:2.23} and \eqref{e:kkk} below) immediately implies the following.

\begin{corollary}\label{C:note333}
Assume that assumption  {\bf (H)} holds and  $\bE [ \mu^p]<\infty $ for some $p>1$.
Let  $\Omega_2\subset \Omega $  be as in Theorem $\ref{T:2.2}$
of full probability measure. Then for every $\w \in \Omega_2$ and every $f\in C_c (\R^d)$,
$U^{\e}_\lambda f$ strongly converges
in $L^2$-spaces to
 $U_\lambda^K f$ as $\e\to0$.
\end{corollary}
The definition of strong convergence in $L^2$-spaces with changing reference measures
can be found in the appendix of this paper.
Corollary \ref{C:note333} can also be proved
by using the
Mosco convergence of Dirichlet
forms -- See Subsection \ref{Mso-s} in the appendix for this alternative approach.

\section{Either assumption {\bf (A)} or   {\bf (B)} implies assumption {\bf (H)}}\label{S:3}

In the section, we will prove that
 either assumption {\bf (A)} or  {\bf (B)} implies assumption {\bf (H)}.
 For this, we first consider the weak convergence of non-local bilinear forms, and then study the compactness of functions with uniformly bounded Dirichlet
forms.
We need the following lemma, which is an extension of Proposition \ref{P:2.1}.

\begin{lemma}\label{L:3.2}
Let $(\Omega, \FF, \Pp)$ be a probability space
on which there is a stationary and ergodic measurable group of  transformations $\{\tau_x\}_{x\in \R^d}$
with $\tau_0 = {\rm id}$.
\begin{itemize}
\item[(i)]  Suppose that $\nu(z;\w)$ is a non-negative measurable function on $\R^d\times \Omega$ such that the function $z\mapsto \Ee
  \left[ \nu(z; \cdot)^p\right]$ is locally integrable for some $p>1$.
 Then there is a subset $\Omega_0 \subset \Omega$ of full probability measure so that for every $\w \in \Omega_0$
 and  every compactly supported $f\in L^q(\R^d\times \R^d; dx\,dy)$ with $q=p/(p-1)$,
\begin{equation}\label{l6-1-0}
\lim_{\varepsilon\to 0}
 \iint_{\R^d\times \R^d} f(x,z)\nu(z;\tau_{x/\varepsilon}\w)\,dz\,dx =\iint_{\R^d\times \R^d} f(x,z) \Ee \left[ \nu(z;\cdot) \right]\,dz\,dx.
\end{equation}

\item[(ii)] Suppose that $\nu_1$ and $\nu_2$ are two non-negative random variables on $( \Omega, \FF,  \bP)$ so that
$\bE \left[ \nu_1^p+\nu_2^p\right]<\infty$ for some $p>1$.
 Then there is a subset $\Omega_0 \subset \Omega$ of full probability measure  so that for every $\w \in \Omega_0$
 and  every compactly supported $f\in L^q(\R^d\times \R^d; dx\,dy)$ with $q=p/(p-1)$,
\begin{equation}\label{e:3.3}
\lim_{\varepsilon\to 0}
 \iint_{\R^d\times \R^d} f(x,y) \nu_1(\tau_{x/\varepsilon}\w) \nu_2(\tau_{y/\varepsilon}\w) \,dx\,dy =
 \Ee \left[ \nu_1 \right] \Ee \left[ \nu_2 \right] \iint_{\R^d\times \R^d} f(x,y) \,dx\,dy.
\end{equation}

\end{itemize}
 \end{lemma}

\begin{proof}
 (i) By the Fubini theorem  and Proposition \ref{P:2.1},
  for any bounded $A,B\in \mathscr{B}(\R^d)$ and a.s.\ $\w\in \Omega$,
\begin{equation}\label{e:3.2}\begin{split}
\lim_{\e\to 0}\iint_{\R^d\times \R^d} \I_{A\times B}(x,z) \nu(z;\tau_{x/\e }\w)\,dz\,dx
&= \lim_{\e \to 0}\int_{A} \left(\int_B \nu(z;\tau_{x/\e})\,dz\right)\,dx
= \int_A\Ee \left[\int_B \nu(z;\cdot)\,dz\right]\,dx   \\
&= \iint_{\R^d\times \R^d}  \I_{A\times B}(x,z) \Ee \left[ \nu(z;\cdot) \right]\, dz\,dx.
\end{split}\end{equation}
The above also holds with $\nu(z; \w)^p$ in place of $\nu(z; \w)$.

Let
\begin{align*}
\mathscr{S}:=\bigg\{f(x,z)=\sum_{i=1}^m a_i\I_{A_i\times
B_i}(x,z):\ m\in \{1,2,\cdots\},\
a_i\in \mathbb{Q},\ A_i,B_i\in \mathscr{B}_{\mathbb{Q}}(\R^d)  \bigg\}.
\end{align*}
Here $\mathbb{Q}$ denotes the set of all rational numbers, and
$\mathscr{B}_{\mathbb{Q}}(\R^d)$ denotes the collection of all bounded
cubes in $\R^d$ whose end points are rational numbers.
Then there is a set $\Omega_0 \subset \Omega$ of full probability measure  so that
for every $\omega \in \Omega_0$,
\eqref{l6-1-0} holds  for every  $f\in \mathscr{S}$, and
 \begin{equation}\label{e:3.5}
\lim_{\e\to 0}\iint_{\R^d\times \R^d} \I_{A\times B}(x,z) \nu(z;\tau_{x/\e }\w)^p \,dz\,dx
= \iint_{\R^d\times \R^d}  \I_{A\times B}(x,z) \Ee \left[ \nu(z;\cdot)^p \right]\, dz\,dx.
\end{equation}

For general compactly supported  $f\in L^q (\R^d\times \R^d;dx\,dy)$, take $A, B\in \mathscr{B}_{\mathbb{Q}}(\R^d)$
so that  $ {\rm supp} [f] \subset A\times B$. Since $C_b(A\times B)$ is dense in $L^q (A\times B;dx\,dy)$
and $\mathscr{S}$ is dense in $C_b(A\times B)$ under the uniform norm,
we can find a sequence
$\{\varphi_n\}_{n\ge1}\subset \mathscr{S}$ such that
$$
\lim_{n \to \infty}\|\varphi_n-f\|_{L^q(A\times B;dx\,dy)}=0.
$$
 Note that  $q=p/(p-1)>1$ is  the conjugate of $p>1$. It follows \eqref{e:3.2} and \eqref{e:3.5}, as well as the fact \eqref{l6-1-0} holds  for every  $f\in \mathscr{S}$ and $\omega \in \Omega_0$, that for every $\w \in \Omega_0$,
\begin{align*}
&  \limsup_{\e\to 0} \left|\iint_{\R^d\times \R^d} f(x,z)\nu(z;\tau_{x/\e }\w)\,dz\,dx-\iint_{\R^d\times \R^d}
f(x,z)\Ee \left[ \nu(z;\cdot)\right]  \,dz\,dx\right|\\
&\le  \limsup_{\e\to 0} \left|  \iint_{\R^d\times \R^d} (f (x,z)-\varphi_n(x,z)) \nu(z;\tau_{x/\e }\w)\,dz\,dx  -\iint_{\R^d\times \R^d}
( f(x,z) -\varphi_n(x,z)
) \Ee \left[ \nu(z;\cdot)\right] \, dz\,dx \right| \\
&\quad    + \limsup_{\e\to 0} \left| \iint_{A\times B} \varphi_n(x,z)\nu(z;\tau_{x/\e }\w)\,dz\,dx-\iint_{A\times B}
\varphi_n(x,z)\Ee \left[ \nu(z;\cdot)\right] \, dz\,dx \right| \\
&\leq   \limsup_{\e\to 0}       \iint_{ A\times B } |f(x,z)-\varphi_n (x, z)| \nu(z;\tau_{x/\e }\w)\,dz\,dx + \iint_{ A\times B }
  |f(x,z)-\varphi_n (x, z)| \Ee \left[ \nu(z;\cdot)\right]  \,dz\,dx\\
   &\leq  \| f-\varphi_n \|_{L^q(A\times B;dx\,dy)}        \left[  \limsup_{\e\to 0}\left( \iint_{A\times B }  \nu(z;\tau_{x/\e }\w)^p\,dz\,dx\right)^{1/p}
   + \left( \iint_{A\times B } \left(  \Ee \left[ \nu(z;\cdot)\right] \right)^p dz\,dx \right)^{1/p}\right]\\
  &\leq      2 \| f-\varphi_n \|_{L^q(A\times B;dx\,dy)}    \left( \iint_{A\times B }  \bE \left[ \nu(z; \cdot )^p\right]\, dz\,dx\right)^{1/p}.
   \end{align*}
By taking $n\to \infty$, we get
$$
\limsup_{\e\to 0} \left|\iint_{\R^d\times \R^d} f(x,z)\nu(z;\tau_{x/\e }\w)\,dz\,dx-\iint_{\R^d\times \R^d}
f(x,z)\Ee \left[ \nu(z;\cdot)\right] \,dz\,dx\right| =0.
$$
This establishes \eqref{l6-1-0}.

(ii) Note that by Proposition \ref{P:2.1} that there is a subset $\Omega_0\subset \Omega$ of full probability measure so that
for every $\w \in \Omega_0$,  any bounded $A,B\in \mathscr{B}(\R^d)$ and $\gamma=1$ or $p$,
\begin{align*}
  \lim_{\e\to 0}\iint_{\R^d\times \R^d} \I_{A\times B}(x,z)  \left( \nu_1 (\tau_{x/\e }\w)\nu_2 (\tau_{y/\e }\w) \right)^\gamma \,dx\,dy
&=\lim_{\e \to 0} \left( \int_{A}  \nu_1 (\tau_{x/\e }\w)^\gamma \,dx \right)  \left( \int_B \nu_2 (\tau_{y/\e }\w)^\gamma \,dy \right)  \\
 &=  \bE \left[\nu_1^\gamma \right] \bE \left[\nu_2^\gamma \right] \, |A\times B|.
\end{align*}
 With this at hand, the assertion can be proved in exactly the same way
as that for (i).
\end{proof}

In the proof of Proposition \ref{P:3.1}(i) below, we need the following maximal ergodic theorem for multiplicative additive processes with continuous parameter.

\begin{proposition}\label{p3-0}
Suppose that $0\le F\in L^1(\Omega;\Pp)$. Then, there is a constant $C>0$ such that for all
 $\lambda,  R_0>0$,
\begin{equation}\label{p3-0-1}
 \Pp\left(\left\{\omega\in\Omega: \sup_{\e\in (0,1)}\int_{[0,R_0]^d}F(\tau_{x/\e}\w)\,dx>\lambda\right\}\right)
\le C R_0^d \, \Ee[F] /\lambda .
\end{equation}
 \end{proposition}

Although the maximal ergodic theorem for multiplicative additive processes has been used in some literature, we can not find a suitable reference for its proof in the continuous parameter setting. For safe of the completeness, here we will provide the

\begin{proof}[Proof of Proposition $\ref{p3-0}$]
Without loss of generality,  we assume $R_0=1$
in \eqref{p3-0-1}.

Let $\mathscr{I}=\{[x,x+k2^m]^d: x\in \Z^d, m,k\in \Z_+\}$,
and define
$$
F_I(\w):=\int_{I}F(\tau_x\w)\,dx,\quad I\in \mathscr{I}.
$$
It is easy to verity $\{F_I(\w)\}_{I\in \mathscr{I}}$ satisfies (2.1)--(2.3) in \cite[Page 201]{Kre} (indeed, (2.2) in \cite[Page 201]{Kre} holds with equality), and so
$\{F_I(\w)\}_{I\in \mathscr{I}}$ is a (discrete) additive process on the integer lattice $\Z^d$.
Let $Q_m:=[0,2^m]^d$ for $m\ge0$,
and
$$
\bar F_*(\w):=\sup_{m\ge 0}\frac{1}{2^{md}}\int_{Q_m}F(\tau_x\w)\,dx.
$$
According to \cite[Page 205, Corollary 2.7]{Kre}, there is a constant $c_0>0$ so that for all $\lambda>0$,
\begin{equation}\label{p3-0-2}
\Pp\left(\bar F_*(\w)>\lambda\right)\le c_0\lambda^{-1}\Ee[F].
\end{equation}

Now, we define
$$
F_*(\w):=\sup_{\e\in (0,1)}\left[\e^d\int_{[0,\e^{-1}]^d}F(\tau_x\w)\,dx\right]=\sup_{\e\in (0,1)}\int_{[0,1]^d}F(\tau_{x/\e}\e)\,dx.
$$
Then, for any $m\ge0$ and $\e\in (2^{-(m+1)},2^{-m}]$,
\begin{align*}
 &\left|2^{-md}\int_{Q_m}F(\tau_x\w)\,dx-\e^d\int_{[0,\e^{-1}]^d}F(\tau_x\w)\,dx
\right| \\
&\le  2^{-md}\int_{Q_{m+1}\backslash Q_m}F(\tau_x\w)\,dx+ (2^{-md}-2^{-(m+1)d})
\int_{Q_{m+1}}F(\tau_x\w)\,dx\\
&\le  (2^{d+1}-1)\frac{\int_{Q_{m+1}}F(\tau_x\w)\,dx}{|Q_{m+1}|}\le (2^{d+1}-1)\bar F_*(\w).
\end{align*}
Hence, for any $m\ge0$,
$$
\sup_{\e\in (2^{-(m+1)},2^{-m}]}\left[\e^d\int_{[0,\e^{-1}]^d}F(\tau_x\w)\,dx\right]
\le 2^{-md}\int_{Q_m}F(\tau_x\w)\,dx+(2^{d+1}-1)\bar F_*(\w)\le 2^{d+1}\bar F_*(\w),
$$
which implies that
$$
F_*(\w)=\sup_{m\ge0}\sup_{\e\in (2^{-(m+1)},2^{-m}]}\left[\e^d\int_{[0,\e^{-1}]^d}F(\tau_x\w)\,dx\right]\le 2^{d+1} \bar F_*(\w).
$$
Therefore, by \eqref{p3-0-2},
we obtain
\begin{align*}
\Pp\left(\left\{\omega\in \Omega: F_*(\w)>\lambda\right\}\right)
\le \Pp\left(\left\{\omega\in \Omega: \bar F_*(\w)>2^{-(d+1)}\lambda\right\}\right)\le
c_02^{d+1} \Ee[F] /\lambda .
\end{align*}
This proves \eqref{p3-0-1}.
\end{proof}

\subsection{Weak convergence of  bilinear forms}\label{Sect3-1}
Recall that
$\Gamma \subset \R^d$ is
an infinite symmetric cone  that has non-empty interior.
When
$d=1$, $\Gamma$ is just $ \R$.

\begin{proposition}\label{P:3.1}
\begin{itemize}
\item [(i)]
Suppose that  {\bf(A1)} holds and that there is a non-negative random variables $\Lambda $
on $(\Omega, \FF, \Pp)$ so that
$\bE [ \Lambda^{p}] < \infty$ for some  $p>1$ and
\begin{equation}\label{e:3.1}
 \kappa (x, y; \w) \leq \Lambda(\tau_x \w) + \Lambda (\tau_y \w)
\quad \hbox{for every } \w \in \Omega,\, x, y \in \R^d.
\end{equation}
Then there is a subset $\Omega_1 \subset \Omega$ of full probability measure  so that for every $\w \in \Omega_1$,
any $\eta>0$, $f\in B(\R^d)$ and $g\in B_c(\R^d)$,
\begin{align*}
&\lim_{\e\to 0}\int_{\R^d}\int_{\{\eta<|z|<1/\eta,\,z\in \Gamma\}}
\frac{(f(x+z)-f(x))(g(x+z)-g(x))}{|z|^{d+\alpha}}\kappa (x/\e ,(x+z)/\e;\w)\,dz\,dx\\
&= \int_{\R^d}\int_{\{\eta<|z|<1/\eta,\,z\in \Gamma\}}\frac{(f(x+z)-f(x))(g(x+z)-g(x))}{|z|^{d+\alpha}}\,(\bar \nu(z)+\bar \nu(-z))\,dz\,dx,
\end{align*} where $\bar \nu$ is a non-negative measurable function given in assumption {\bf (A1)}.

\item [(ii)]
Let $\nu_1 $ and $\nu_2 $ be non-negative random variables on $(\Omega, \FF, \Pp)$ such that
$\bE \left[ \nu_1^{p} + \nu_2^{p} \right] < \infty$ for some $p>1$.
Then there is a subset $\Omega_2 \subset \Omega$ of full probability measure  so that for every $\w \in \Omega_2$,
any $\eta>0$,  $f\in B(\R^d)$ and $g\in B_c(\R^d)$,
\begin{align*}
&\lim_{\e\to 0}\int_{\R^d}\int_{\{\eta<|x-y|<1/\eta,\,x-y\in \Gamma\}}
\frac{(f(y)-f(x))(g(y)-g(x))}{|x-y|^{d+\alpha}}
\nu_1 (\tau_{x/\e} \w) \nu_2 (\tau_{y/\e} \w) \, \,dx\,dy\\
&= \int_{\R^d}\int_{\{\eta<|x-y|<1/\eta,\,x-y\in \Gamma\}}\frac{(f(y)-f(x))(g(y)-g(x))}{|x-y|^{d+\alpha}}\,
\bE [\nu_1] \bE [\nu_2]
\,dx\,dy.
\end{align*}
\end{itemize}
\end{proposition}

\begin{proof}
 We first prove the assertion (ii), and then  (i).

(ii)
Note that for every
$\eta>0$, $f\in B(\R^d)$ and $g\in B_c(\R^d)$,
$$
F(x, y):= \I_{\{\eta<|x-y|<1/\eta,\,x-y\in \Gamma\}} \frac{(f(x)-f(y))(g(x)-g(y))}{|x-y|^{d+\alpha}}
$$
is a bounded and compactly supported function on $\R^d\times \R^d$. Since
\begin{align*}
&\int_{\R^d}\int_{\{\eta<|x-y|<1/\eta,\,x-y\in \Gamma\}}
\frac{(f(x)-f(y))(g(x)-g(y))}{|x-y|^{d+\alpha}}
\nu_1(x/\e ;\w)\nu_2(y/\e;\w)\,dx\,dy\\
&=2 \iint_{\R^d\times \R^d} F(x, y) \nu_1(\tau_{x/\e }\w)
\nu_2(\tau_{y/\e}\w)\,dx\,dy.
\end{align*}
 Proposition \ref{P:3.1}(ii) follows directly from  Lemma \ref{L:3.2}(ii).

(i) Let $F(x,y)$ be  the bounded and compactly supported function defined above. Note that
\begin{align*}
&\int_{\R^d}\int_{\{\eta<|z|<1/\eta,\,z\in \Gamma\}}
\frac{(f(x+z)-f(x))(g(x+z)-g(x))}{|z|^{d+\alpha}}
\nu\left(z/\e;\tau_{x/\e}\w\right)\,dz\,dx\\
&=  \iint_{\R^d\times \R^d} F(x,x+z) \nu\left(z/\e;\tau_{x/\e}\w\right)\,dz\,dx,
\end{align*}
and by Proposition \ref{P:2.1} and condition \eqref{e:3.1}, for $\bP$-a.s. $w\in \Omega$,
 \begin{align*}
\limsup_{\e \to 0}\iint_{A\times B}  \nu\left(z/\e;\tau_{x/\e}\w\right)^p\,dz\,dx
&\leq  2^p \limsup_{\e \to 0}\iint_{A\times B} \left(\Lambda (\tau_{x/\e} \w)^p +
\Lambda (\tau_{(x+z)/\e} \w)^p \right)  \,dz \,dx \\
& =  2^{p} (|A\times B| +|(A+B) \times B|)\, \bE \left[ \Lambda^p \right]
<\infty,
\end{align*} where $A+B=\{x+y\in \R^d: x\in A, y \in B\}$.
Note also that  by {\bf(A1)} and  \eqref{e:3.1},
$$
\sup_{x\in \R^d} \bE [ \nu (x; \cdot)] \leq \sup_{x\in \R^d} \bE [  \kappa (x, 0; \cdot) ]\leq\bE \left[ \Lambda\right]+
\sup_{x\in \R^d}\bE\left[\Lambda (\tau_x(\cdot)) \right] = 2\bE  [\Lambda ]
\leq 2 (\bE  [\Lambda^p ])^{1/p} <\infty
$$
and so, by \eqref{a2-1-2a}, the function $\bar \nu$ in {\bf(A1)} is bounded.
Proposition \ref{P:3.1}(i) can be proved in exactly the same way as that of Lemma \ref{L:3.2}(i)
once one can verify that there is a subset $\Omega_2\subset \Omega$ of full probability measure
so that for every $\w \in \Omega_2$ and any $A, B\in \mathscr{B}_\mathbb{Q} (\R^d)$,
\begin{equation}\label{e:step1}
\lim_{\e \to 0}\iint_{A\times B}  \nu(z/\e;\tau_{x/\e }\w)\,dz\,dx=
\iint_{A\times B} \bar \nu(z)\,dz\,dx.
\end{equation} Furthermore, by \eqref{a2-1-2a}, in order to verify \eqref{e:step1} it suffices to prove that there is a subset $\Omega_3\subset \Omega_2$ of full probability measure
so that for  every $\w \in \Omega_3$ and any $A, B\in \mathscr{B}_\mathbb{Q} (\R^d)$,
\begin{equation}\label{e:step1-}
\lim_{\e \to 0}
\Big|\iint_{A\times B}(\nu(z/\e;\tau_{x/\e }\w)-\Ee[ \nu(z/\e;\cdot)])
\,dz\,dx\Big|=0.
\end{equation}
For notational convenience,  we  will prove the above for
  $A=B=[0,1]^d$. The proof for the general case is similar.

Recall that $\nu_k(z;\w):=\nu(z;\w)\wedge k$.
Fix $\delta>1$. For any $\e>0$,   choose $m\ge 1$ such that
$\delta^{m-1}\leq 1/\e < \delta^m$.
Then,
\begin{align*}
&\Big|\iint_{[0,1]^d\times [0,1]^d}\big(\nu_k(z/\e;\tau_{x/\e }\w)-\Ee [\nu_k(z/\e;\cdot)]\big)\,
dz\,dx\Big|\\
&=\Big|\e^{2d}\iint_{[0,\e^{-1}]^d\times [0,\e^{-1}]^d}
\big(\nu_k(z;\tau_{x}\w)-\Ee [\nu_k(z;\cdot)]\big)\,
dz\,dx\Big|\\
&\le  \delta^{-2(m-1)d}\Big|\iint_{[0,\delta^{m-1}]^d\times [0,\delta^{m-1}]^d}
\big(\nu_k(z;\tau_{x}\w)-\Ee [\nu_k(z;\cdot)] \big)\,
dz\,dx\Big|\\
& \quad+ \delta^{-2(m-1)d} \Big|\iint_{ [0,\e^{-1}]^d\times [0,\e^{-1}]^d \setminus [0,\delta^{m-1}]^d\times [0,\delta^{m-1}]^d}
\big(\nu_k(z;\tau_{x}\w)-\Ee [\nu_k(z;\cdot)] \big)\,
dz\,dx\Big|\\
&  \le  \Big| \delta^{-2(m-1)d} \iint_{[0,\delta^{m-1}]^d\times [0,\delta^{m-1}]^d}
\big(\nu_k(z;\tau_{x}\w)-\Ee [\nu_k(z;\cdot)] \big)\,
dz\,dx\Big|\\
 &\quad +\delta^{-2(m-1)d}\iint_{[0,\delta^m]^d\times
[0,\delta^m]^d\backslash[0,\delta^{m-1}]^d\times
[0,\delta^{m-1}]^d}\nu_k(z;\tau_x \w)\,dz\,dx\\
& \quad + \delta^{-2(m-1)d}\iint_{[0,\delta^m]^d\times
[0,\delta^m]^d\backslash[0,\delta^{m-1}]^d\times
[0,\delta^{m-1}]^d}\Ee [\nu_k(z;\cdot)] \,dz\,dx\\
&=: \sum_{i=1}^3 I_i^m.
\end{align*}

We consider $I_1^m$ first. For any $m\in \Z_+$ and  $x=(x^{(1)},x^{(2)},\cdots,x^{(d)})\in
\delta^{-m}\Z^d$,    set
$$Q_x^m=\prod_{1\le i\le d}
[x^{(i)},x^{(i)}+\delta^{-m}].
$$
Define
\begin{align*}
\mathscr{S}_m =& \{x\in \delta^{-m}\Z^d: Q_x^m\subset [0,1]^d\}, \\
\mathscr{T}_m =& \{x\in \delta^{-m}\Z^d: Q_x^m\cap [0,1]^d\neq \emptyset,
Q_x^m\cap([0,1]^d)^c\neq \emptyset\}, \\
 F_m(\w) =& \int_{[0,1]^d}\nu_k (\delta^m z;\w )\,dz.
 \end{align*}
 We have
\begin{equation}\label{l2-2-4}
\begin{split}
 &\delta^{-2md}\iint_{[0,\delta^m]^d\times [0,\delta^m]^d} \nu_k(z;\tau_{x}\w)\,dz\,dx\\
&=\iint_{[0,1]^d\times [0,1]^d}
\nu_k (\delta^m z;\tau_{\delta^m x}\w )\,dz\,dx=\int_{[0,1]^d}F_m (\tau_{\delta^m x}\w )\,dx\\
&=\sum_{i: x_i\in \mathscr{S}_m }\int_{Q_{x_i}^m}F_m (\tau_{\delta^m x}\w )\,dx
+\sum_{j: x_j\in \mathscr{T}_m}\int_{Q_{x_j}^m\cap [0,1]^d}F_m (\tau_{\delta^m x}\w )\,dx\\
&=:L_1^m+L_2^m.
\end{split}
\end{equation}
For  $x\in \delta^{-m}\Z^d$, let
$$
 F_{x,m}(\w)=\int_{Q_x^m}F_m (\tau_{\delta^m z} \w )\,dz .
 $$
 Since $0\le \nu_k(z;\cdot)\le k$ for all $z\in \R^d$,
 there is a constant $c_1(k)>0$ (which may depend on $k$) such that
\begin{equation}\label{l2-2-5}
\sup_{x\in \delta^{-m}\Z^d}{\rm Var} (F_{x,m} ) \leq \sup_{x\in \delta^{-m}\Z^d}\Ee \left[ F_{x,m}^2 \right]
 \le c_1(k)\delta^{-2md}.
\end{equation}
Hence,
 \begin{align*}
 {\rm Var} (L_1^m) &={\rm Var} \Big(  \sum_{i: x_i\in \mathscr{S}_m}   F_{x_i,m}\Big)=\sum_{n=0}^{[{\delta^m}]}
 \sum_{i,j:\, x_i,x_j \in
\mathscr{S}_m,
d
n\delta^{-m}\le |x_i-x_j|\le d
(n+1)\delta^{-m}}
{\rm Cov } ( F_{x_i,m},  F_{x_j,m} ) \\
&=\sum_{i,j:\,  x_i,x_j \in \mathscr{S}_m,|x_i-x_j|\le d\delta^{-m}}
 {\rm Cov } (  F_{x_i,m},  F_{x_j,m} )  \\
&\quad +\sum_{n=1}^{[\delta^m]}
\sum_{i,j: \, x_i,x_j \in
\mathscr{S}_m,  dn\delta^{-m}< |x_i-x_j|\le d(n+1)\delta^{-m}}  {\rm Cov } ( F_{x_i,m},   F_{x_j,m})  \\
&=:J_1^m+J_2^m.
\end{align*}
Note that $|\{(i,j):x_i,x_j\in \mathscr{S}_m, |x_i-x_j|\le d\delta^{-m}\}|\le c_2\delta^{md}$ for some constant $c_2>0$. Then, by \eqref{l2-2-5} and the
Cauchy-Schwarz inequality,
$$
|J_1^m|\le c_3(k)\delta^{-md}.
$$
On the other hand, it follows from \eqref{a2-1-1} that when $dn\delta^{-m}<|x_i-x_j|\le d(n+1)\delta^{-m}$
for some $n\ge 1$,
\begin{align*}
  {\rm Cov } (  F_{x_i,m},   F_{x_j,m} ) &=
\int_{Q_{x_i}^m}\int_{Q_{x_j}^m}   {\rm Cov } \left( F_m  (\tau_{\delta^m x}(\cdot)),
  F_m(\tau_{\delta^m y}(\cdot) ) \right) \,dy\,dx\\
&=\int_{Q_{x_i}^m}\int_{Q_{x_j}^m}    {\rm Cov } \left(
F_m (\tau_{\delta^m x}(\cdot)) ,    F_m (\tau_{\delta^m (y-x)}\circ \tau_{\delta^m x}(\cdot ) )\right)\,dy\,dx\\
&\le c_4(k)
\int_{Q_{x_i}^m}\int_{Q_{x_j}^m} \left|\delta^m
(y-x)\right|^{-l}\,dy\,dx  \le c_5(k) n^{-l}\delta^{-2md},
\end{align*}
where  the first inequality
follows  from \eqref{a2-1-1}, and the second one is due to  \eqref{l2-2-5} and
the fact that $|x-y|\ge c_6n\delta^{-m}$ for all
$x\in Q_{x_i}^m$ and $y\in Q_{x_j}^m$ with $|x_i-x_j|>
dn\delta^{-m}$.
Since
$$|\{(i,j): x_i,x_j\in \mathscr{S}_m, dn\delta^{-m} <
|x_i-x_j|\le d(n+1)\delta^{-m}\}|\le c_{7}\delta^{md}n^{d-1}$$ for
some constant $c_{7}>0$ independent of $m$ and $n$, we get
$$
|J_2^m|\le
c_{8}(k)\delta^{-md}\sum_{n=1}^{[\delta^m]}n^{-(l+1-d)}\le
c_{9}(k)m\delta^{-m(d\wedge l)}\log \delta.
$$
Therefore, according to these two estimates
for $J_1^m$ and $J_2^m$,
we obtain
$$
 {\rm Var} (L_1^m)
 \le c_{10}(k)m\delta^{-m(d\wedge l)}\log \delta.
$$
In particular, by the Markov inequality,
for any
$\gamma>0$,
$$
\sum_{m=1}^\infty\Pp\Big(\Big|\sum_{i: x_i\in \mathscr{S}_m}
F_{x_i,m}-\sum_{i: x_i\in \mathscr{S}_m}\Ee [F_{x_i,m}] \Big|^2>\gamma\Big) <\infty.$$ Hence, by
the Borel-Cantelli lemma, for a.s.\ $\w\in \Omega$,
\begin{align*}
\lim_{m \to \infty} L_1^m
&=\lim_{m \to \infty}\sum_{i: x_i\in \mathscr{S}_m}
 F_{x_i,m}=\lim_{m \to \infty}\sum_{i: x_i\in \mathscr{S}_m} \Ee[ F_{x_i,m}]
 =\lim_{m \to \infty}\int_{\cup_{i: x_i\in \mathscr{S}_m}Q_{x_i}^m}\int_{[0,1]^d}\Ee [\nu_k(\delta^m z;\cdot)]\,dz\,dx
 \\
&=\lim_{m \to \infty}\int_{[0,1]^d\times [0,1]^d}\Ee [\nu_k(\delta^m z;\cdot)]\,dz\,dx
 =\lim_{m\to \infty} \delta^{-2md}\int_{[0,\delta^m]^d\times [0,\delta^m]^d}\Ee [\nu_k(z;\cdot)]\,dz\,dx,
\end{align*}
where the fourth equality we have used the fact $\lim_{m \to
\infty}\cup_{i: x_i\in \mathscr{S}_m}Q_{x_i}^m=[0,1]^d$.
Note that
$$|\{x_i\in \delta^{-m}\Z^d: Q_{x_i}^m\cap [0,1]^d\neq \emptyset, Q_{x_i}^m \cap  ([0,1]^d)^c\neq \emptyset\}|\le
c_{11}\delta^{m(d-1)}.$$
This along with \eqref{l2-2-5} and the
Cauchy-Schwarz
inequality gives us
\begin{align*}
\Ee[|L_2^m|^2]&\le |\{x_i\in \delta^{-m}\Z^d:
Q_{x_i}^m\cap [0,1]^d\neq \emptyset, Q_{x_i}^m \cap ([0,1]^d)^c \neq \emptyset\}|^2
\sup_{x_i\in \delta^{-m}\Z^d}\Ee[F_{x_i,m}^2]\le c_{12}(k)\delta^{-2m},
\end{align*}
which in turn implies that
for any $\gamma>0$,
$$
\sum_{m=1}^\infty\Pp\big(|L_2^{m}|>\gamma\big)<\infty.
$$
Hence  by the Borel-Cantelli lemma,     $\lim_{m \to \infty}L_2^{m}=0$ a.s.

Combining  both estimates for $L_1^m$ and $L_2^m$ with \eqref{l2-2-4} yields that
$$\lim_{m\to\infty} \delta^{-2md}\iint_{[0,\delta^m]^d\times [0,\delta^m]^d} \nu_k(z;\tau_{x}\w)\,dz\,dx = \lim_{m\to \infty} \delta^{-2md}\int_{[0,\delta^m]^d\times [0,\delta^m]^d}\Ee [\nu_k(z;\cdot)]\,dz\,dx;$$ that is,
\begin{align*}
\lim_{m \to \infty}I_1^m=&
\lim_{m \to \infty} \delta^{-2(m-1)d} \left| \iint_{[0,\delta^{m-1}]^d\times [0,\delta^{m-1}]^d}
\left(\nu_k(z;\tau_{x}\w)-\Ee [\nu_k(z;\cdot)]\right)\,dz\,dx\right|=0.
\end{align*}

By
 the argument for $I_1^m$
(in particular, by applying the Borel-Cantelli lemma),
we have
that for a.s.\ $\w\in \Omega$ and every $k\ge1$, there exists a constant $m_0(k,\w)>0$ such that
for every $m\ge m_0(k,\w)$,
\begin{align*}
 \iint_{[0,\delta^m]^d\times [0,\delta^m]^d \backslash [0,\delta^{m-1}]^d\times [0,\delta^{m-1}]^d}\nu_k(z;\tau_x\w)\,dz\,dx
&\le
\iint_{[0,\delta^m]^d\times [0,\delta^m]^d \backslash [0,\delta^{m-1}]^d\times [0,\delta^{m-1}]^d}(\Ee [\nu_k(z;\cdot)] +1)\,dz\,dx\\
&\le c_{13}(k)(\delta^{2md}-\delta^{2(m-1)d})=c_{14}(k)\delta^{2md}(1-\delta^{-2d}).
\end{align*}
In particular, for a.s.\ $\w\in \Omega$,
$$ \limsup_{m \rightarrow \infty}I_2^m \le c_{14}(k)(\delta^{2d}-1).$$
Since
$0\le \nu_k(z;\w)\le k$,
it is obvious that
$$
 \limsup_{m \rightarrow \infty} I_3^m  \le c_{15}(k)(\delta^{2d}-1).
$$

Putting all the above estimates together, we  conclude
that for every fixed $\delta>1$ and $k\ge1$, there exists a
$\Pp$-null set $\mathcal{N}_{\delta,k}$ such that for every $\w\in \Omega \setminus \mathcal{N}_{\delta,k}$,
\begin{equation}\label{l2-2-6}
\begin{split}
&\limsup_{\e \to 0}
\Big|\iint_{[0,1]^d\times [0,1]^d} (\nu_k(z/\e;\tau_{x/\e }\w )-\Ee  [\nu_k (z/\e;\cdot)])
\,dz\,dx\Big|\le \limsup_{m \to \infty}
\sum_{i=1}^3 I_i^m \le c_{16}(k)(\delta^{2d}-1).
\end{split}
\end{equation}

Let  $\mathcal{N}:=\cup_{n=1}^\infty \cup_{k=1}^\infty\mathcal{N}_{1+{1}/{n},k}$ and  $\Omega_2:=\Omega \setminus \mathcal{N}$.
Clearly $\Pp(\Omega_2)=1$,
and \eqref{l2-2-6} holds for every $\delta=1+1/n$, $k\ge1$ and $\w\in \Omega_2$.
Letting $n\to \infty$,  we have
for every $\w \in \Omega_2$,
\begin{equation}\label{l2-2-7}
\lim_{\e \to 0}
\Big|\iint_{[0,1]^d\times [0,1]^d}(\nu_k(z/\e;\tau_{x/\e }\w)-\Ee[ \nu_k(z/\e;\cdot)])
\,dz\,dx\Big|=0\quad\hbox{ for every } k\in \Z_+.
\end{equation}

For $k\ge1$, let
\begin{align*}
A_k:=\left\{\w\in \Omega:\sup_{\e\in (0,1)}\left|\iint_{[0,1]^d\times [0,1]^d}\nu_{2^k}(z/\e;\tau_{x/\e }\w)\,dz\,dx-
\iint_{[0,1]^d\times [0,1]^d}\nu(z/\e;\tau_{x/\e }\w)\,dz\,dx\right|>1/k\right\}.
\end{align*}According to \eqref{assu-1} and \eqref{e:3.1},
\begin{align*}
\Pp\left(A_k\right) &\le \Pp\left(\left\{\omega\in \Omega:
\sup_{\e\in (0,1)}\iint_{[0,1]^d\times [0,1]^d}\nu(z/\e;\tau_{x/\e }\w)
\I_{\{\nu(z/\e;\tau_{x/\e }\w)>2^k\}}\,dz\,dx>1/ k\right\}\right)\\
&\le \Pp\left(\left\{\omega \in \Omega: \sup_{\e\in (0,1)}\iint_{[0,1]^d\times [0,1]^d}\frac{\nu(z/\e;\tau_{x/\e }\w)^p}{2^{k(p-1)}}\,dz\,dx>1/k\right\}\right)\\
&\le \Pp\left(\left\{\omega \in \Omega:\sup_{\e\in (0,1)}\iint_{[0,1]^d\times [0,1]^d} \left(\Lambda(\tau_{x/\e}\w)+\Lambda(\tau_{(x+z)/\e}\w)\right)^p
\,dz\,dx>k^{-1}2^{k(p-1)} \right\}\right)\\
&\le \Pp\left(\left\{\omega\in \Omega: \sup_{\e\in (0,1)}
 \int_{[0,2]^d}
\Lambda(\tau_{x/\e }\w)^p \,dx>2^{-p} k^{-1}{2^{k(p-1)}} \right\}\right).
\end{align*} Here $p>1$ is given in the assumption of Proposition \ref{P:3.1}(i),
and in the last inequality we used the facts that $(a+b)^p\le 2^{p-1}(a^p+b^p)$ for $a,b\ge0$ and that
$$
 \iint_{[0,1]^d\times [0,1]^d}\Lambda(\tau_{(x+z)/\e }\w)^p\,dz\,dx \le
   \int_{[0,2]^d}
 \Lambda(\tau_{x/\e}\w)^p\,dx .
$$
Hence by Proposition \ref{p3-0} ,
$$
\Pp\left(\left\{\omega\in \Omega: \sup_{\e\in (0,1)}
  \int_{[0,2]^d}
\Lambda(\tau_{x/\e }\w)^p \,dx>2^{-p}k^{-1}{2^{k(p-1)}}\right\}\right)
\le c_{17}k2^{-k(p-1)}\Ee[\Lambda^p],
$$
where $c_{17}>0$ is independent of $k$. Putting all the estimates above together, we get $\sum_{k=1}^\infty \Pp(A_k)<\infty$. By the Borel-Cantelli lemma again, we can find
a subset $\Omega_3\subset \Omega_2$ with full probability measure such that for every $\w\in \Omega_3$, there exists
 $k_0(\w)>0$ such that
\begin{align*}
\left|\iint_{[0,1]^d\times [0,1]^d}\nu_{2^k}(z/\e;\tau_{x/\e }\w)\,dz\,dx-
\iint_{[0,1]^d\times [0,1]^d}\nu(z/\e;\tau_{x/\e }\w)\,dz\,dx\right|\le 1/k
\end{align*} for every  $k\ge k_0(\w)$ and $\e\in (0,1)$.

Combining this with \eqref{l2-2-7} yields (first letting $\e \to 0$ and then $k\to \infty$) that
for every $\w\in \Omega_3$,
\begin{align*}
\lim_{\e \to 0}\left|\iint_{[0,1]^d\times [0,1]^d}\nu(z/\e;\tau_{x/\e }\w)\,dz\,dx-
\iint_{[0,1]^d\times [0,1]^d}\Ee  [\nu (z/\e;\cdot)]\,dz\,dx\right|=0.
\end{align*}
Thus \eqref{e:step1-} holds with $A=B=[0, 1]$  for every $\w\in \Omega_3$.
This completes the   the proof of  assertion (i) in Proposition \ref{P:3.1}.
\end{proof}

\subsection{Pre-compactness  in $L^1$-spaces}

In this part, we give the compactness for a sequence of uniformly bounded functions
whose associated scaled Dirichlet forms are also uniformly bounded. The following is the
main result of this subsection.

\begin{proposition}\label{P:compact} Suppose that either {\bf(A2')} or
 {\bf(B2')}  holds. Then there is a subset $\Omega_0\subset \Omega$ of full probability measure  so that, for  every $\w \in \Omega_0$ and
   any   collection
of functions $\{f_\e:\e\in (0,1]\}$  with  $f_{\e}\in
\FF^{\e,\w}$ for any $\e\in (0,1]$ having
 $$
 \limsup_{\e\to0}  \left(\|f_\e\|_\infty+ \EE^{\e,\w}(f_\e,f_\e)\right)<\infty,
$$
$\{f_\e: \e\in (0,1]\}$ is  pre-compact as $\e \to 0$ in $L^1(B(0,r);dx)$ for all $r>1$.
\end{proposition}

\begin{lemma}\label{l5-2}
Suppose that either {\bf(A2')} or   {\bf(B2')}
 holds. Then there is subset $\Omega_0\subset \Omega$ of full probability measure  so that the following holds
 for every $\w \in \Omega_0$.  Suppose that  $\{f_\e:\e\in (0,1]\}$  is a collection
of functions  with  $f_{\e}\in
\FF^{\e,\w}$ for  $\e\in (0,1]$  and
$$
\limsup_{\e \to 0}\EE^{\e,\w}(f_\e,f_\e)<\infty.
$$ Then, for every $r>1$ and $0<|h|\le r/3$,
\begin{equation}\label{l5-2-1}
\limsup_{\e\to 0} \sup_{x_0 \in B(0, r)}\int_{B_{2h}(x_0,r)}|f_\e(x+he_i)-f_\e(x)|\,dx\le c_0(r) h^{\alpha/2}\,\limsup_{\e\to 0}\E^{\e,\w}(f_\e,f_\e)^{1/2},\quad 1\le i \le d,
\end{equation}
where $\{e_i: 1\le i \le d\}$ is the orthonormal basis of $\R^d$,
$B_{2h}(x_0,r):=\{y\in B(x_0,r):|y-\partial B(x_0,r)|>2h\}$ and $c_0(r)$ is a positive constant depending on $r$ but independent of $h$, $f$, $\e$ and $\w$.
\end{lemma}

\begin{proof}
It suffices to prove \eqref{l5-2-1} for every fixed $e_i$ and the case that $h>0$. The argument below is
partially motivated by the proof of the compact embeddings in fractional Sobolev spaces; see
\cite[Theorem 4.54, p.\ 216]{DD}.

For
any $z\in \R^d$,  denote by $(r(z),\theta(z)) \in \R_+\times \Ss^{d-1}$
 its spherical coordinate.
Let $\Theta := \Gamma \cap \Ss^{d-1}$, which has non-empty interior.
Hence there are a non-empty open set $\wt  \Theta \subset \Theta$  and a constant $N\geq 1$ large enough so that
$$
 \left\{\theta\Big(z-\frac{e_i}{N}\Big): 1\le r(z)\le 2, \theta(z)\in \wt
\Theta\right\} \subset \Theta.
$$
Consequently,  $\theta(z-he_i)=\theta(z/(hN)-e_i/N)\in \Theta$ for any $z=(r(z),\theta(z))$ with $Nh\le r(z)\le 2Nh$ and $\theta(z)\in
\wt  \Theta$. Let
$$
G_{h,\wt  \Theta}:=\left\{z\in \R^d: Nh\le r(z)\le 2Nh,
\theta(z)\in \wt  \Theta \right\} \quad \hbox{and} \quad
G_{h,  \Theta}:= \left\{z\in \R^d: Nh\le r(z)\le 2Nh,
\theta(z)\in    \Theta \right\} .
$$
Clearly $G_{h,\wt  \Theta}\subset G_{h, \Theta}$ and $| G_{h,\wt  \Theta}| \asymp |G_{h, \Theta}| \asymp h^d$.

For every $x_0\in B(0,r)$, $0<h<r/3$ and $f_\e$ given in the statement, we have
\begin{equation}\label{e:3.10}\begin{split}
&\int_{B_{2h}(x_0,r)}|f_\e(x+he_i)-f_\e(x)|\,dx   \\
&\leq  c_1 h^{-d} \int_{B_{2h}(x_0,r)} \int_{x+G_{h,\wt  \Theta}}|f_\e(x+he_i)-f_\e(x)|\,dz\, dx  \\
&\le c_1 h^{-d} \left( \int_{B_{2h}(x_0,r)} \int_{x+G_{h,\wt  \Theta}}|f_\e(x+he_i)-f_\e(z)|\,dz\,dx
+\int_{B_{2h}(x_0,r)} \int_{x+G_{h,\wt  \Theta}}|f_\e(x)-f_\e(z)|\,dz\,dx \right)  \\
&=: c_1 h^{-d} \left( I_{x_0,h,1}(f_\e)+I_{x_0,h,2}(f_\e) \right).
\end{split}\end{equation}
By a change of variables,
\begin{equation}\label{l5-2-4}
\begin{split}
I_{x_0,h,1}(f_\e)
&\le\int_{B(0, 2r)}\int_{G_{_{h,\wt  \Theta}}-he_i}|f_\e(x+z+he_i)-f_\e(x+he_i)|\,dz\,dx\\
 &\le  \int_{B(0, 2r)} \int_{G_{h,\Theta}}|f_\e(x+z+he_i)-f_\e(x+he_i)|\,dz\,dx\\
 &\le \int_{B(0,3r)}  \int_{G_{h,\Theta}}|f_\e(x+z)-f_\e(x)|\,dz\,dx.
 \end{split}
\end{equation} Similarly, we have
$$
I_{x_0,h,2}(f_\e)\le \int_{B(0,2r)}   \int_{G_{h,\Theta}}|f_\e(x+z)-f_\e(x)|\,dz\,dx.$$

 (i) Assume that {\bf (A2')} holds. By \eqref{l5-2-4}, the H\"older inequality and Proposition \ref{P:2.1},
 \begin{align*}
\limsup_{\e \to 0} \sup_{x_0\in B(0, r)} I_{x_0,h,1}(f_\e)&\le
\limsup_{\e \to 0} \Big(\int_{B(0,3r)}\int_{G_{h,\Theta}}\frac{(f_\e(x+z)-f_\e(x))^2}{|z|^{d+\alpha}}\Lambda_1(\tau_{x/\e }
\w)\,dz\,dx
\Big)^{1/2}\\
&\quad\times \limsup_{\e \to 0} \Big(\int_{B(0, 3r)}\int_{G_{h,\Theta}}|z|^{d+\alpha}\Lambda_1(\tau_{x/\e } \w)^{-1}\,dz\,dx\Big)^{1/2}\\
&\le c_2 \limsup_{\e \to 0} \bigg[\bigg(\int_{B(0,3r)}\Lambda_1(\tau_{x/\e } \w)^{-1}\,dx\bigg)^{1/2}\E^{\e,\w}(f_\e,f_\e)^{1/2}\bigg]
   h^{d+{\alpha}/{2}}\\
&= c_2    h^{d+{\alpha}/{2}}  \left( |B(0, 3r)| \bE[ \Lambda_1^{-1}] \right)^{1/2}    \limsup_{\e \to 0}  \E^{\e,\w}(f_\e,f_\e)^{1/2}   \\
&=c_3   h^{d+{\alpha}/{2}} \limsup_{\e \to 0}  \E^{\e,\w}(f_\e,f_\e)^{1/2} ,
\end{align*}
where  $c_2, c_3$ are positive constants independent of $x_0$, $h$ and
$\e$, but may depend on $N$ and $r$.

Similarly,  we have
$$
\limsup_{\e \to 0} \sup_{x_0\in B(0, r)}  I_{x_0,h,2}(f_\e)\le
c_4 h^{d+{\alpha}/{2}} \limsup_{\e\to 0}\E^{\e,\w}(f_\e,f_\e)^{1/2} ,
$$
where $c_4$ is a positive constant independent of $x_0$, $h$
and $\e$. Thus we have by \eqref{e:3.10},
$$
\limsup_{\e\to 0}\sup_{x_0\in B(0, r)} \int_{B_{2h}(x_0,r)}|f_\e(x+he_i)-f_\e(x)|\,dx\le
c_5 h^{\alpha/2}  \limsup_{\e\to 0}\E^{\e,\w}(f_\e,f_\e)^{1/2},
$$
which establishes \eqref{l5-2-1}.

(ii) Next we  assume that
 {\bf (B2')} holds.
According to \eqref{l5-2-4} and the H\"older inequality, for any $0<h<r/3$,
\begin{equation}\label{l5-2-5}
\begin{split}
& \limsup_{\e \to 0} \sup_{x_0\in B(0, r)} I_{x_0,h,1}(f_\e) \\
&\le
\limsup_{\e \to 0} \left(\int_{B(0,3r)}\int_{G_{h,\Theta}}\frac{(f_\e(x+z)-f_\e(x))^2}{|z|^{d+\alpha}}\Lambda_1(\tau_{x/\e }
\w) {\Lambda_1}  (\tau_{(x+z)/\e}\w)\,dz\,dx
\right)^{1/2}\\
&  \quad\times
\limsup_{\e \to 0} \left(\int_{B(0,3r)}\int_{G_{h,\Theta}}|z|^{d+\alpha}\Lambda_1(\tau_{x/\e }
\w)^{-1}
{\Lambda_1}(\tau_{(x+z)/\e}\w)  ^{-1}\,dz\,dx\right)^{1/2}\\
&\le
c_7h^{(d+\alpha)/2} \limsup_{\e \to 0}  \left( \E^{\e,\w}(f_\e,f_\e)  \int_{B(0,3r)}\int_{x+G_{h,\Theta}} \Lambda_1(\tau_{x/\e }
\w)^{-1} {\Lambda_1}  (\tau_{y/\e}\w)^{-1}\,dy\,dx   \right)^{1/2}.
\end{split}
\end{equation}
Here $c_7$ is a
positive constant  independent of
$h$ and $\e$.

Below, denote
$$
K_h:=\left\{ (x, y): x \in B( 0, 3
r), y\in x+ G_{h,\Theta} \right\}\subset \R^d\times \R^d.
$$
If we directly apply Lemma \ref{L:3.2}(ii) to estimate $\iint_{K_h}\Lambda_1(\tau_{x/\e}\w)^{-1}
\Lambda_1(\tau_{y/\e}\w)^{-1}\,dx\,dy$, then it requires that $\Lambda_1^{-1}\in L^p(\Omega;\Pp)$ for some
$p>1$. Instead we will adopt a different approach under the weak condition $\Lambda_1^{-1}\in L^1(\Omega;\Pp)$ as in \eqref{a2-2-2} of assumption {\bf(B2')}. For every $0<h\le r/3$, we
can find $\{x_i\}_{i=1}^m\subset \R^d$ such that $B(0,3r)\subset \cup_{i=1}^m B(x_i,h)$ and
$\sum_{i=1}^m |B(x_i,h)|\le c_0|B(0,3r)|$, where $c_0$ is independent of $h,r$, $m$ and $\{x_i\}_{i=1}^m$, but the integer $m$ and
$\{x_i\}_{i=1}^m$ may depend on $h$. It is easy to verify that
$$
K_h \subset \bigcup_{i=1}^m B(x_i,h)\times (B(x_i,(2N+1)h)\setminus B(x_i,(N-1)h)).
$$
Hence,
\begin{align*}
&\lim_{\e \to 0}\iint_{K_h}\Lambda_1(\tau_{x/\e}\w)^{-1}
\Lambda_1(\tau_{y/\e}\w)^{-1}\,dx\,dy\\
&\le
\lim_{\e\to 0}\sum_{i=1}^m \int_{B(x_i,h)}\int_{B(x_i,(2N+1)h)\backslash B(x_i,(N-1)h)}
\Lambda_1(\tau_{x/\e}\w)^{-1}
\Lambda_1(\tau_{y/\e}\w)^{-1}\,dx\,dy\\
&=\sum_{i=1}^m\left(\lim_{\e \to 0}\int_{B(x_i,h)}\Lambda_1(\tau_{x/\e}\w)^{-1}\,dx\right)
\cdot \left(\lim_{\e \to 0}\int_{B(x_i,(2N+1)h)\setminus B(x_i,(N-1)h)}\Lambda_1(\tau_{y/\e}\w)^{-1}\,dy\right)\\
&=\sum_{i=1}^m \Ee[\Lambda^{-1}]^2\cdot |B(x_i,h)|\cdot |B(x_i,(2N+1)h)\backslash B(x_i,(N-1)h)|\le
c_8h^d \sum_{i=1}^m |B(x_i,h)|\le c_9h^d,
\end{align*}
where $c_8$ and $c_9$ are positive constants independent of
$h$ (but may depend on $N$ and $r$). Here, in the second equality above we have used
Proposition \ref{P:2.1} (so only $\Lambda_1^{-1}\in L^1(\Omega;\Pp)$ is required), and the last inequality is due to the fact $\sum_{i=1}^m |B(x_i,h)|\le c_0|B(0,3r)|$.
Therefore, combining this estimate with \eqref{l5-2-5}, we find that
$$ \limsup_{\e \to 0} \sup_{x_0\in B(0, r)} I_{x_0,h,1}(f_\e) \le
c_{10} \Big(\limsup_{\e\to 0} \E^{\e,\w}(f_\e,f_\e)^{1/2}\Big)h^{d+{\alpha}/{2}}.
$$

Similarly, we
can obtain
$$
\limsup_{\e \to 0}\sup_{x_0\in B(0, r)}   I_{x_0,h,2}(f_\e)\le
c_{11} \Big(\limsup_{\e\to 0} \E^{\e,\w}(f_\e,f_\e)^{1/2}\Big)h^{d+{\alpha}/{2}}.
$$
Thus we have by \eqref{e:3.10},
$$
\limsup_{\e\to 0} \sup_{x_0\in B(0, r)} \int_{B_{2h}(x_0,r)}|f_\e(x+he_i)-f_\e(x)|\,dx\le
c_{12} h^{\alpha/2}  \limsup_{\e\to 0}\E^{\e,\w}(f_\e,f_\e)^{1/2},
$$
which establishes \eqref{l5-2-1}.
\end{proof}

\begin{proof}[Proof of Proposition $\ref{P:compact}$] We first claim that for every $r>1$
and $h\in \R^d$ with $|h|<r/3$
\begin{equation}\label{l5-2-1-00}
\limsup_{\e\to 0}
\int_{B_{2|h|}(0,r)}|f_\e(x+h)-f_\e(x)|\,dx\le
c_0(r;\w) h^{\alpha/2}.
\end{equation}
 Indeed, writing $h=(h^{(1)},\cdots, h^{(i)},\cdots, h^{(d)})\in \R^d$, we have for any $0\le i\le d$,
\begin{align*}&\int_{B_{2|h|}(0,r)}
|f_\e(x+(h^{(1)},\cdots, h^{(i)},h^{(i+1)},0,\cdots, 0))-f_\e(x+(h^{(1)},
\cdots, h^{(i)},0,\cdots, 0))|\,dx\\
&=\int_{B_{2|h|}((h^{(1)},\cdots, h^{(i)},0,\cdots, 0),r)}|f_\e(x+(0,\cdots,
0,h^{(i+1)},0,\cdots, 0))-f_\e(x)|\,dx\\
&\le \int_{B_{2|h^{(i+1)}|}((h^{(1)},\cdots, h^{(i)},0,\cdots, 0),r)}|f_\e(x+(0,
\cdots, 0,h^{(i+1)},0,\cdots, 0))-f_\e(x)|\,dx. \end{align*}
Here, we set
$(h^{(1)},\cdots, h^{(i)},0,\cdots, 0)=0$  when $i=0$. This along with \eqref{l5-2-1} gives \eqref{l5-2-1-00}.

 On the other hand, for any $\delta>0$,
$$
\limsup_{\e \to 0} \int_{B(0,r)\setminus {B_{2\delta}(0,r)}} |f_\e(x)|\,dx
\le \big(\limsup_{\e\to 0 }\|f_\e\|_\infty\big)|B(0,r)\setminus {B_{2\delta}(0,r)}|\le c_1(r)\delta,
$$
where $c_1(r)$ is a positive constant independent of $\delta$ and $\e$.

Therefore, for every $r>1$
and $\zeta>0$, there exists a constant
$\delta:=\delta(r,\zeta;\w)$ such that for every $h\in \R^d$ with $|h|<\delta$,
\begin{equation}\label{l5-2-1-000}\limsup_{\e\to 0}
\int_{B_{\delta}(0,r)}|f_\e(x+h)-f_\e(x)|\,dx\le\zeta\end{equation} and
\begin{equation}\label{l5-2-1-11}
\limsup_{\e \to 0} \int_{B(0,r)\setminus {B_{\delta}(0,r)}} |f_\e(x)|\,dx\le \zeta.
\end{equation}

It then follows from \eqref{l5-2-1-000}, \eqref{l5-2-1-11}  and \cite[Theorem 1.95, p.\ 37]{DD}
  that $\{f_\e: \e\in (0,1]\}$ is
pre-compact as $\e \to 0$ in $L^1 (B(0,r);dx)$ for all $r>1$. The proof is complete.
  \end{proof}

\medskip

\subsection{Proofs of Theorems \ref{thm:mainth1} and \ref{T:1.5}, and the assertion of Example \ref{exa2}} \label{S:3.3}

\begin{proof}[Proof  of Theorem  $\ref{thm:mainth1}$]
According to Theorems \ref{T:2.2} and
\ref{C:note3},
we only need to verify that assumption {\bf(A)} implies assumption {\bf(H)}.
By Propositions \ref{P:3.1} and \ref{P:compact},
  assumptions {\bf(A)}  implies properties (i) and (iv) in assumption {\bf(H)}.
So it remains to show that properties (ii) and (iii) in assumption {\bf(H)} hold  as well.

Suppose assumption {\bf(A2)} holds. There is a subset $\Omega_1\subset \Omega$
so that Proposition \ref{P:2.1} holds with $\Lambda_2$ in place of $\nu$.
For   $g\in C_c^1 (\R^d)$,  let $R_0>1$ be such  that ${\rm supp} [g] \subset B(0, R_0)$.
For  every $\w \in \Omega_1$ and $\eta \in (0, 1/(2R_0))$,
\begin{equation}\label{e:ppp}\begin{split}
& \limsup_{\e \to 0} \iint_{\{|x-y|\le \eta\}}(g(x)- g(y))^2\frac{\kappa ({ x/\e},y/\e;\w)}{|x-y|^{d+\alpha}}\,dx\,dy\\
&\leq 2  \limsup_{\e \to 0}  \iint_{\{|x-y|\le \eta\}}(g(y)- g(x))^2\frac{\Lambda_2\big(\tau_{x/\e }\w\big)}{|x-y|^{d+\alpha}}\,dx\,dy   \\
&\leq  2 \| \nabla g\|_\infty^2 \limsup_{\e \to 0}  \int_{B(0,R_0+\eta)}  \Big( \int_{\{|z|\le \eta\} }  \frac{1}{|z|^{d+\alpha-2}}\,dz \Big)
\Lambda_2\big(\tau_{x/\e }\w\big) dx\\
&\le  c_1\eta^{2-\alpha} | B(0,R_0+\eta)| \bE [\Lambda_2].
\end{split}
\end{equation}
 In particular, we have for every  $\w\in \Omega_1$,
$$
\lim_{\eta \to 0}\limsup_{\e \to 0}\iint_{\{|x-y|\le \eta\}}(g(x)- g(y))^2\frac{\kappa ({ x/\e},y/\e;\w)}{|x-y|^{d+\alpha}}\,dx\,dy=0.
$$

On the other hand,  noting that $1/\eta >2R_0$, we have
\begin{equation}\label{e:3.16}\begin{split}
&\limsup_{\e \to 0} \iint_{\{|x-y|>1/\eta\}}(g(y)- g(x))^2\frac{\kappa \big(\tau_{x/\e }\w,  \tau_{y/\e }\w; \w \big)}{|x-y|^{d+\alpha}}\,dx\,dy
\\
&\leq  2 \limsup_{\e \to 0}  \iint_{\{|x-y|>1/\eta\}}(g(y)- g(x))^2\frac{\Lambda_2\big(\tau_{x/\e }\w\big)}{|x-y|^{d+\alpha}}\,dx\,dy \\
& \leq  2 \limsup_{\e \to 0}  \int_{B(0, R_0)} \left( \int_{\{|y-x|>1/\eta\}} \frac{1}{|y-x|^{d+\alpha}} dy \right) g(x)^2
 \Lambda_2\big(\tau_{x/\e }\w\big) \,dx \\
&\quad  + 2 \limsup_{\e \to 0}  \int_{B(0, R_0)^c} \left( \int_{\{y\in B(0, R_0): |y-x|>1/\eta\}} \frac{g(y)^2}{|y-x|^{d+\alpha}} dy \right)
 \Lambda_2\big(\tau_{x/\e }\w\big) \,dx  \\
&\le c_2\| g\|^2_\infty  |B(0, R_0)| \eta^{\alpha/2} \left( \eta^{\alpha/2} \bE [\Lambda_2]  +   \limsup_{\e \to 0}
 \int_{B(0, R_0)^c}  \frac{ \Lambda_2\big(\tau_{x/\e }\w\big) }{  | x|^{d+\alpha/2}}   \,dx  \right) .  \end{split}
\end{equation}

Recall that, by Proposition \ref{P:2.1},
for every $\w \in \Omega_1$,
$$
\lim_{\e\to 0}\e^d\int_{B(0,1/\e)}\Lambda_2(\tau_{x}\w)\,dx=
\lim_{\e\to 0} \int_{B(0,1)}
\Lambda_2(\tau_{x/\e }\w)\,dx  = |B(0,1)|
\bE [ \Lambda_2].
$$
Hence, for every $\w\in \Omega_1$, there exists  $\e_0(\w)>0$ such that
\begin{equation}\label{t1-4-1}
\int_{B(0,1/\e)}\Lambda_2(\tau_{x}\w)\,dx\le 2\bE [ \Lambda_2]|B(0,1)|\e^{-d},\quad \e\in (0,\e_0(\w)).
\end{equation}
Next, we choose $k_0:=k_0(\w)\ge1$ such that
$2^{k_0}\ge  \e_0(\w)^{-1} $.
  Thus, for every   $\w \in \Omega_1$,
\begin{align*}
 \limsup_{\e \to 0}
 \int_{B(0, R_0)^c}  \frac{ \Lambda_2\big(\tau_{x/\e }\w\big) }{  | x|^{d+\alpha/2}}    \,dx
 & \le   \limsup_{\e \to 0}  \int_{B(0,2^{k_0})\backslash B(0,R_0)} \frac{ \Lambda_2\big(\tau_{x/\e }\w\big) }{  | x|^{d+\alpha/2}}    \,dx
 +\limsup_{\e \to 0}\sum_{k=k_0+1}^\infty
 \int_{Q_k}  \frac{ \Lambda_2\big(\tau_{x/\e }\w\big) }{  | x|^{d+\alpha/2}}    \,dx
 \\
 &=:\limsup_{\e \to 0}I_1^\e +\limsup_{\e \to 0}I_2^\e,
 \end{align*} where $Q_k=\{x\in \R^d: 2^{k-1}< |x|\le 2^k\}$.
According to Proposition \ref{P:2.1} again, we obtain
\begin{align*}
\limsup_{\e \to 0}I_1^\e=\int_{B(0,2^{k_0})\backslash B(0,R_0)}\frac{\Ee[\Lambda_2]}{|x|^{d+\alpha/2}}\,dx\le\int_{ B(0,R_0)^c}\frac{\Ee[\Lambda_2]}{|x|^{d+\alpha/2}}\,dx<\infty.
\end{align*}
Meanwhile, by \eqref{t1-4-1}, it holds that
 \begin{align*}
 \limsup_{\e \to 0}I_2^\e &\leq     \limsup_{\e \to 0}  \sum_{k=k_0+1}^\infty   2^{-(k-1)(d+\alpha/2)}
 \int_{B(0,2^k)}    \Lambda_2\big(\tau_{x/\e }\w\big)    \,dx \\
 &=\limsup_{\e \to 0}  \sum_{k=k_0+1}^\infty   2^{-(k-1)(d+\alpha/2)}
 \e^d\int_{B(0,2^k/\e)}    \Lambda_2\big(\tau_{x}\w\big)    \,dx\\
 &
 \le 2 |B(0,1)|\Ee[\Lambda_2]\limsup_{\e \to 0}\sum_{k=k_0+1}^\infty  2^{-(k-1)(d+\alpha/2)}\e^d (2^k/\e)^d\\
 &=c_4\sum_{k=k_0+1}^\infty 2^{-k\alpha/2}\le c_4\sum_{k=1}^\infty 2^{-k\alpha/2}
 <\infty.
\end{align*}
Combining all the estimates above with \eqref{e:3.16} shows that for every $\w \in \Omega_1$,
$$
\lim_{\eta\to0}
\limsup_{\e \to 0}
\iint_{\{|x-y|\ge 1/\eta\}}(g(x)- g(y))^2\frac{\kappa ({ x/\e},y/\e;\w)}{|x-y|^{d+\alpha}}\,dx\,dy=0;
$$
that is,  property (ii) of  assumption {\bf (H)} holds.

Evidently,  there is $\Omega_2\subset \Omega_1$ of full probability measure  so that
 the conclusion of Proposition \ref{P:2.1} holds with $\Lambda_2^p$ in place of $\nu$.
In view of \eqref{a2-2-1} of assumption {\bf (A2)},  for every $\w \in \Omega_2$, by H\"older's inequality,
 \begin{align*}
  \limsup_{\e \to 0} \int_{B(0, R)} \left( \int_{B(0, R)}  \kappa ({x/\e},y/\e;\w) \,dy  \right)^p\, dx
&\leq c_5 R^{(p-1)d/p}   \limsup_{\e \to 0} \iint_{B(0, R)\times B(0, R)}  \kappa ( {x/\e},y/\e;\w)^p \, dx\,dy   \\
&\leq  2^{p+1} c_5 R^{(p-1)d/p}    \limsup_{\e \to 0} \int_{B(0, R)}  \int_{B(0, R)}  \Lambda_2( \tau_{x/\e} \w )^p \, dx\,dy   \\
&= 2^{p+1} c_5 R^{(p-1)d/p}   |B(0, R)|^2 \bE [ \Lambda_2^p ]<\infty.
\end{align*}
Hence property  (iii) of assumption {\bf (H)}  holds as well. This shows that assumption {\bf (A2)} yields properties (ii) and (iii) of  assumption {\bf (H)},
and so proves Theorem \ref{thm:mainth1}.
\end{proof}

\medskip

\begin{proof}[Proof  of Theorem  $\ref{T:1.5}$]
By Propositions \ref{P:3.1} and \ref{P:compact},
 assumption  {\bf (B)} implies properties (i) and (iv) in assumption {\bf(H)}.
So it remains to show that properties (ii) and (iii) in assumption {\bf(H)} hold as well.
Note that under  assumption {\bf(B2)},
  for any $\w\in \Omega$ and $x,y\in \R^d$,
$$\kappa(x,y;\w)\le \Lambda_2(\tau_x\w)\Lambda_2(\tau_y\w)\le \frac{1}{2}(\Lambda_2(\tau_x\w)^2+\Lambda_2(\tau_y\w)^2).$$
With this at hand, we can follow the argument
in the proof of Theorem  \ref{thm:mainth1}
 to show that assumption {\bf (B2)} implies  property (ii) of  assumption {\bf (H)}.   Furthermore,
an analogous argument also shows that  assumption {\bf (B2)} implies  property (iii) of  assumption {\bf (H)}.

\smallskip

 We point out that the second moment condition on $\Lambda_2$  in assumption {\bf (B2)} is only used in establishing \eqref{e:2.3}    (and  \eqref{eq:domain}).
 The remaining
properties in (ii)
and (iii) of assumption {\bf (H)} hold under a weaker assumption that $\bE [ \Lambda_2^p ] <\infty$ for some $p>1$.
In particular, \eqref{e:2.3-0} follows from the corresponding argument in the proof of Theorem \ref{thm:mainth1} by using Lemma \ref{L:3.2}(ii) instead of Proposition \ref{P:2.1}.
\end{proof}

\medskip

\begin{proof}[Proof of the assertion in Example $\ref{exa2}$]
 Define
 \[
 \mu:=\frac 1{Z_\mu}\frac{\lambda_2}{\lambda_1},  \quad \hbox{where  }
 Z_\mu:=\Ee \left[{\lambda_2}/{\lambda_1}\right].
 \]
 Clearly, $\Ee [ \mu]=1$, and
we can  rewrite   \eqref{e:prod-ex} as
$$
 \LL^{\e, \w} f(x)=\frac 1{Z_\mu  \mu ( \tau_{x/\e} \omega)}\,  {\rm p.v.} \int_{\R^d}  (f(y)-f(x))
 \frac{\lambda_2( \tau_{x/\e} \w)\lambda_2(\tau_{y/\e}\w)}{|y-x|^{d+\alpha}}
 \I_{\Gamma}(y-x) \,dy.
 $$
Thus, the operator $\LL^{\e, \w}$ is symmetric in $L^2(\R^d;\mu( \tau_{x/\e} \w)\,dx)$, and  is associated with
the symmetric  Dirichlet form $(\EE^{\e,\w}, \FF^{\e, \w})$ on $L^2(\R^d; \mu(\tau_{x/\e} \w) \,dx)$  given by
$$
\EE^{\e,\w}(f,g)=\frac{1}{2 Z_\mu}\iint_{\R^d\times \R^d} (f(x)-f(y))(g(x)-g(y))\frac{\lambda_2(\tau_{x/\e} \w)\lambda_2(\tau_{y/\e} \w)}{|x-y|^{d+\alpha}}\I_{\Gamma}(y-x) \,dx\,dy.
$$
    Under assumptions of the example, we see that \eqref{e:1.17} and \eqref{e:1.16} are satisfied with
  $$ \nu_1 = \nu_2 = \frac1{  \sqrt{ Z_\mu }} \lambda_2,
  $$
and $\bE [ \mu^p]<\infty$. Hence, by Proposition \ref{P:1.4}, assumption {\bf (B)} is fulfilled. Therefore, the conclusion of the example follows readily  from  Theorem \ref{T:1.5}.
\end{proof}

\section{Appendix}\label{Appdx}\label{S:4}

\subsection{Proofs of Lemma \ref{L:d1} and Proposition \ref{P:1.4}}

\begin{proof}[Proof of Lemma $\ref{L:d1}$] Let $\LL^\w$ and $(P_t^{\w})_{t\ge0}$ be the generator and the semigroup
of the process $X^{\w}$, respectively. Let $\LL^{\e,\w}$ and  $(P_t^{\varepsilon,\w})_{t\ge0}$ be the generator and the semigroup of the process $X^{\varepsilon,\w}$, respectively. Then, for any $f\in C_c^\infty(\R^d)$,
$$
P_t^{\varepsilon,\w}f(x)=\Ee_x^{\w}
[ f (\varepsilon X_{t/\varepsilon^{\alpha}}^{\w})]=P_{t/\varepsilon^{\alpha}}^{\w} f^{(\varepsilon)}(x/\varepsilon),$$ where $f^{(\varepsilon)}(x)=f(\varepsilon x)$. Thus,
$$\LL^{\e,\w}f(x)=\frac{d P_t^{\varepsilon,\w}f(x)}{d t}\Big|_{t=0}=
\varepsilon^{-\alpha} \LL^{\w} f^{(\varepsilon)}(x/\varepsilon).
$$

For any $f,g\in C_c^\infty(\R^d)$, by the facts that the operator $\LL^\w$ is symmetric on $L^2(\R^d;\mu^\w(dx))$ and $$-\int_{\R^d} \LL^\w f(x) g(x)\,\mu^{\w}(dx)= \EE^\w(f,g),$$ we find that
\begin{align*}
&-\int_{\R^d} \LL^{\e,\w}f(x)g(x)\,\mu^{\e,\w}(dx)\\
&= -
\int_{\R^d} \e^{-\alpha}\LL^\w f^{(\e)}(x/\e )g(x)\mu( \tau_{x/\e} \w)\,dx\\
&=-\int_{\R^d} \e^{-\alpha}\LL^\w f^{(\e)}(x/\e )g^{(\e)}(x/\e )\mu( \tau_{x/\e} \w)\,dx\\
&=-{\e^{d-\alpha}}\int_{\R^d}\LL^\w f^{(\e)}(x)g^{(\e)}(x)\mu(\tau_x\w)\,dx\\
&=\frac{\e^{d-\alpha}}{2}\iint_{\R^d\times \R^d\backslash\Delta}\big(f^{(\e)}(y)-f^{(\e)}(x)\big)\big(g^{(\e)}(y)-g^{(\e)}(x)\big)\frac{\kappa (x,y;\w)}{|x-y|^{d+\alpha}}\I_{\{x-y\in \Gamma\}}\,dx\,dy\\
&=\frac{\e^{d-\alpha}}{2}\iint_{\R^d\times \R^d\backslash\Delta}\big(f(\e y)-f(\e x)\big)\big(g(\e y)-g(\e x)\big)\frac{\kappa (x,y;\w)}{|x-y|^{d+\alpha}}\I_{\{x-y\in \Gamma\}}\,dx\,dy\\
&=\frac{1}{2}\iint_{\R^d\times \R^d\backslash\Delta}\big(f(y)-f(x)\big)\big(g(y)-g(x)\big)\frac{\kappa\big(x/\e ,y/\e;\w\big)}{|x-y|^{d+\alpha}}\I_{\{x-y\in \Gamma\}}\,dx\,dy.
\end{align*}In particular,
$$ \int_{\R^d} \LL^{\e,\w}f(x)g(x)\,\mu^{\e,\w}(dx)= \int_{\R^d} \LL^{\e,\w}g(x)f(x)\,\mu^{\e,\w}(dx).$$ Hence, the desired assertion follows.
\end{proof}

\begin{proof}[Proof of Proposition $\ref{P:1.4}$]

(i) Suppose that \eqref{e:1.16} holds. Clearly \eqref{t2-1-**} and \eqref{a2-2-2} hold by taking
 $\Lambda_1 (\w) :=  \sqrt{2 \nu_1 (\w)  \nu_2 (\w) }$
 and $\Lambda_2 (\w) := \nu_1 (\w)+ \nu_2 (\w) $.

(ii) Conversely,  suppose that
there are non-negative random variables $\Lambda_1\leq \Lambda_2 $ on $(\Omega, \FF,\Pp)$ so that \eqref{t2-1-**} and \eqref{a2-2-2}
hold.
Taking $x=y=0$ in \eqref{a2-2-2} yields from \eqref{e:1.17} that
\begin{equation}\label{e:4.1}
\Lambda_1^2 (\w) \leq 2 \nu_1 (\w) \nu_2 (\w) \leq \Lambda_2 (\w)^2.
\end{equation}
This in particular implies that
 $$
\bE [ (\nu_1\nu_2)^{-1/2}]  \leq 2^{1/2} \bE [ \Lambda_1^{-1}] < \infty.
$$
We claim that there is a  constant $C >0$ so that
\begin{equation} \label{e:4.2}
  \nu_i (\w ) \leq C  \Lambda_2  (\w) \quad \bP \hbox{-a.s. for } i=1, 2.
\end{equation}
This together with \eqref{e:4.1} will imply that \eqref{e:1.16} holds.
 Since $\bP( 0 < \Lambda_1 \leq \Lambda_2 <\infty) =1$,  there are constants $0<a<b$ so that
$\bP(\Omega_1 )>3/4$, where
$$
\Omega_1 := \left\{\w \in \Omega:  a\leq \Lambda_1 (\w)\leq  \Lambda_2 (\w)\leq b \hbox{ and } a\leq \min \{\nu_1 (\w), \nu_2 (\w)\}
\leq \max \{\nu_1 (\w), \nu_2 (\w)\}\leq b \right\}.
$$
Define
$$
A =\{\w \in \Omega: \tau_x \w \in \Omega_1 \hbox{ for some } x\in \R^d \}.
$$
Clearly, $\Omega_1 \subset A  $ and $\tau_x A  \subset A $ for every $x\in \R^d$. Since  $\tau_x \circ \tau_y=\tau_{x+y}$ for every $x, y\in \R^d$ and $\tau_0={\rm id}$, we have $\tau_x A =A $ for every $x\in \R^d$.  Hence $\bP(A )=1$  as the family of shift operators $\{\tau_x; x\in \R^d\}$ is ergodic.
By  \eqref{e:1.17} and \eqref{a2-2-2},
$$
\nu_1(\tau_x\w)\nu_2(\tau_y\w) +\nu_1(\tau_y\w)\nu_2(\tau_x\w)
\leq \Lambda_2(\tau_x\w)\Lambda_2(\tau_y\w).
$$
For every $\w \in A $, take $x=0$ and  $y\in \R^d$ so that $\tau_y \w \in \Omega_1$ in the above display.   Then this implies that
$$
\nu_1 ( \w ) + \nu_2 (\w)  \leq (b/a) \Lambda_2 (  \w).
$$
This proves the claim  \eqref{e:4.2}, and hence completes the proof  of the  Proposition.
  \end{proof}

\subsection{Mosco convergence of  Dirichlet forms}\label{Mso-s}

In this part, we will study the Mosco convergence for the Dirichlet
forms $(\E^{\e,\w},\F^{\e,\w})$, which yields the strong convergence
of the associated semigroups and resolvents in $L^2$-senses along
any fixed sequence $\{\e_n\}_{n\ge1}$ with $\e_n \to 0$ as $n\to \infty$.

For this, we first recall some known results from \cite{KS, Ko1,Ko2} on Mosco convergence with changing reference measures, and adapt them to our setting. Consider the Hilbert spaces $L^2(\R^d;\mu^{\e,\w}(dx))$ and $L^2(\R^d;dx)$ as these in the present paper. For simplicity, we drop the parameter $\w$, and write $L^2(\R^d; \mu^{\e_n,\w}(dx))$ as $L^2(\R^d;\mu_{\e_n}(dx))$. Then, by Proposition \ref{P:2.1} and $\Ee [\mu]=1$,  there is $\Omega_0\subset \Omega$ of full probability measure so that for all $\w \in \Omega_0$  and $f\in C_c^\infty(\R^d)$,
$$\lim_{n\to\infty}\|f\|_{L^2(\R^d;\mu_{\e_n}(dx))}=\|f\|_{L^2(\R^d;dx)};$$ that is, \emph{$L^2(\R^d;\mu_{\e_n}(dx))$ converges to $L^2(\R^d;dx) $} in the sense of \cite{KS}, see \cite[p.\ 611]{KS} or \cite[Definition 2.1]{Ko2}.

Following  \cite[Definitions 2.4 and 2.5]{KS} or \cite[Definitions 2.2 and 2.2]{Ko2},
we say that a sequence of functions $\{f_n\}_{n\ge1}$ with $f_n\in L^2(\R^d;\mu_{\e_n}(dx))$ for all $n\ge1$
\emph{strongly converges in $L^2$-spaces}
to $f\in L^2(\R^d;dx)$, if there exists a sequence of functions $\{g_m\}_{m\ge1}\subset C_c^\infty(\R^d)$ such that
\begin{equation}\label{e:kkk}\lim_{m\to\infty}\|g_m-f\|_{L^2(\R^d;dx)}=0, \quad \lim_{m\to\infty}\limsup_{n\to\infty} \|g_m-f_n\|_{L^2(\R^d;\mu_{\e_n}(dx))}=0;\end{equation} we say that a sequence of functions $\{f_n\}_{n\ge1}$ with $f_n\in L^2(\R^d;\mu_{\e_n}(dx))$ for all $n\ge1$
\emph{weakly converges in $L^2$-spaces}
to $f\in L^2(\R^d;dx)$, if for every sequence $\{g_n\}_{n\ge1}$  with $g_n\in L^2(\R^d;\mu_{\e_n}(dx))$ for all $n\ge1$ and strongly convergent to $g\in L^2(\R^d;dx)$,
$$\lim_{n\to\infty} \langle f_n, g_n\rangle_{L^2(\R^d;\mu_{\e_n}(dx))}= \langle f, g\rangle_{L^2(\R^d;dx)}.$$ It is obvious that strong convergence in $L^2$-spaces is stronger than weak convergence in $L^2$-spaces; see \cite[Lemma 2.1(4)]{KS}.

For any $\e>0$, let $(\EE^{\varepsilon,\w},
\FF^{\e,\w})$ be a regular Dirichlet form on
$L^2(\R^d;\mu^{\e,\w}(dx))$ given by \eqref{eq:niebiws}, and $(
\E^K,  \F^K)$ be a regular Dirichlet form on $L^2(\R^d;dx)$ given by
\eqref{eq:DeDFK}.   Following \cite[Definition 2.11]{KS}, a sequence
of Dirichlet forms $\{(\E^{\e_n},\F^{\e_n})\}_{n\ge1}$ on
$L^2(\R^d;\mu_{\e_n}(dx))$ is said to be \emph{Mosco convergent} to a
Dirichlet form $(\E^K,\F^K)$ on $L^2(\R^d;dx)$, if
\begin{itemize}
\item[(i)] for every sequence $\{f_n\}_{n\ge1}$ with $f_n\in L^2(\R^d;\mu_{\e_n}(dx))$ for all $n\ge1$ and
converging weakly to $f\in L^2(\R^d;dx)$,
$$
\liminf_{n\to\infty} \E^{\e_n}(f_n,f_n)\ge \E^K(f,f).
$$

\item[(ii)] for any $f\in L^2(\R^d;dx)$, there is a sequence
$\{f_n\}_{\ge1}$ with $f_n\in L^2(\R^d;\mu_{\e_n}(dx))$ for all $n\ge1$ and converging strongly to $f$ such that
$$\limsup_{n\to\infty}\E^{\e_n}(f_n,f_n)\le \E^K(f,f).$$
\end{itemize}

In the above
definition, we have extended the definition of
$\E^{\e_n} (f, f)$ and $\E^K (f, f)$ by
$\E^{\e_n}(f,f)=\infty$ for $f\in L^2(\R^d;\mu_{\e_n}(dx))\backslash \FF^{\e_n}$ and $\E^K(f,f)=\infty$ for $f\in L^2(\R^d;dx)\backslash \FF^K$,
respectively.

\begin{remark}\label{R:4.2} \rm
\begin{itemize}
\item[(i)] The condition (i) in the definition of Mosco convergence holds
true, if for every sequence $\{f_n\}_{n\ge1}$ with $f_n\in
L^2(\R^d;\mu_{\e_n}(dx))$ for all $n\ge1$ and converging weakly to $f\in
L^2(\R^d;dx)$, then $\liminf_{n\to\infty} \E^{\e_n}(f_n,f)\ge
\E^K(f,f).$ See \cite[p.\ 726, the proof Theorem 4.7]{CKK} for the
proof.

 \item[(ii)] \ Note that $C_c^\infty(\R^d)\subset \F$ is a core of both
Dirichlet forms $(\E^{\e_n},\F^{\e_n})$ and $(\E^K,\F^K)$. According
to \cite[Lemma 2.8]{Ko2}, the condition (ii) in the definition of
Mosco convergence holds true, if and only if for any $f\in
C_c^\infty(\R^d)$, $\lim_{n\to\infty} \E^{\e_n}(f,f)=\E^K(f,f).$ See
\cite[Theorem 2.3]{BBCK} or \cite[Lemma 8.2]{CKK} for the case of
the Mosco convergence without changing reference measures.
\end{itemize}
\end{remark}

\begin{theorem}\label{T:mos} Under assumption {\bf(H)},
there is $\Omega_0\subset \Omega$ of full probability measure so that for any $\w\in \Omega_0$ and any $\{\e_n\}_{n\ge 1}$ with $\e_n\to 0$ as $n\to\infty$, Dirichlet forms
$(\E^{\e_n,\w},\F^{\e_n,\w})$  convergence to
$(\E^K,\F^K)$ in the sense of Mosco as $\lim_{n\to\infty}\e_n=0$.\end{theorem}
\begin{proof} Throughout the proof, we drop the
parameter $\w$.
According to (ii) and (iv) of assumption {\bf(H)} and Proposition \ref{P:2.1}, we know that there is $\Omega_0\subset \Omega$ of full probability measure so that for any $\w\in \Omega_0$ and $f\in C_c^\infty(\R^d)$,
$$
\lim_{n \to \infty}\EE^{\e_n}(f,f)=\EE^K(f,f).
$$
By Remark \ref{R:4.2}(ii),   property (ii) in the
definition of Mosco convergence holds true. So it remains  to verify property  (i).

 Let $\{f_n\}_{n\ge1}$ be such that $f_n\in L^2(\R^d;\mu_{\e_n}(dx))$ for any $n\ge1$ and
converging weakly to $f$ in $L^2(\R^d;dx)$. Without loss of
generality, we assume that
$\E^{\e_n}(f_n,f_n)$ converges and
\begin{equation}\label{t6-2-1}
\sup_{n\ge1} \big( \E^{\e_n}(f_n,f_n)+\|f_n\|_{L^2(\R^d;\mu_{\e_n}(dx))}^2 \big) <\infty,
\end{equation}
thanks to \cite[Lemma 2.3]{KS}.
For any $N>0$, let $f^N(x):=(-N)\vee\left(f(x)\wedge N\right)$. Note
that $\E^{\e_n}(f^N_n,f^N_n)\le \E^{\e_n}(f_n,f_n)$.
According to \eqref{t6-2-1}, the fact $\|f_n^N\|_\infty\le N$ and (i) in assumption {\bf(H)},
$\{f_n^N\}_{n\ge 1}$ is a pre-compact set in $L_{loc}^1(\R^d;dx)$.
So there exist a measurable function $f^{*,N}\in L_{loc}^1(\R^d;dx)$ and a
subsequence $\{f_{n_k}\}_{k\ge1}$ of $\{f_n\}_{n\ge1}$ such that for all $r\ge1$,
$$
\lim_{k\to\infty}\|f_{n_k}^N-f^{*,N}\|_{L^{1}(B(0,r);dx)}=0.
$$
Since $\|f_n^N\|_\infty\le N$ for all $n\ge1$, $\|f^{*,N}\|_\infty\le N$ and so for any $r\ge1$, $N>0$ and
$1<q<\infty$,
\begin{equation}\label{e:note}
\lim_{k\to\infty}\|f_{n_k}^N-f^{*,N}\|_{L^{q}(B(0,r);dx)}=0.
\end{equation}
Moreover, it is easy to see that $f^{*,N}(x)=f^{*,M}(x)$ for all $N\le M$ and a.e.\ $ x\in \R^d$ with
$|f^{*,N}(x)|\le N$. Therefore, we can find a measurable function $f_0\in L^1_{loc}(\R^d;dx)$ such that $f_0^N(x)=f^{*,N}(x)$ for
a.e.\ $x\in \R^d$ and $N>0$. Combining this with \eqref{e:note} and the fact that $f_n$ converges weakly to $f$ yields
that $f_0=f$ a.e. and \eqref{e:note} holds with $f^{*,N}$ replaced by $f^N$. That is,
for any $r\ge1$, $N>0$ and
$1<q<\infty$,
\begin{equation}\label{e:note-}
\lim_{k\to\infty}\|f_{n_k}^N-f^N\|_{L^{q}(B(0,r);dx)}=0.
\end{equation}
For notational simplicity, in the following we denote the subsequence $\{f_{n_k}\}_{k\ge1}$ by
$\{f_n\}_{n\ge1}$.

We first suppose that $f\in C_c^\infty(\R^d)$. Take $N>\|f\|_\infty$; that is, $f(x)=f^N(x)$. Then,
following the argument for \eqref{e:oopp} in the proof of Theorem \ref{T:2.2} line
by line,
$$\lim_{n\to\infty}\E^{\e_n}(f_n^N,f)=\E^K(f,f),$$ which along with Remark
\ref{R:4.2}(i)
gives us that
$$\lim_{n\to\infty}\E^{\e_n}(f_n,f_n)\ge \lim_{n\to\infty}\E^{\e_n}(f^N_n,f^N_n)\ge\E^K(f,f).$$

For the general case, we partly follow the proof of \cite[Theorem
4.7]{CKK}. Without loss of generality, we assume that $f\in \FF^K$. For fixed $\gamma>0$, we choose $f_*\in
C_c^\infty(\R^d)$ such that
$$\E^{K}(f-f_*,f-f_*)+\|f-f_*\|_{L^2(\R^d;dx)}^2\le \gamma^2.$$ For any $\eta\in (0,1)$, denote by $$
\E^{\e_n}_{\eta}(f,f)=\frac{1}{2}\iint_{\{\eta<|z|<1/\eta,\,z\in \Gamma\}}(f(x+z)-f(x))^2\frac{\kappa (0,z/\e_n ;\tau_{x/\e_n }\w)}{|z|^{d+\alpha}}\,dz\,dx$$
and
$$\LL^{n}_{\eta}f(x)=\int_{\{\eta<|z|<1/\eta,\, z\in \Gamma\}}(f(x+z)-f(x))\frac{\kappa (0,z/\e_n ;\tau_{x/\e_n }\w)}{|z|^{d+\alpha}}\,dz.$$
Similarly, we can define $\E_{\eta}^K(f,f)$ and $\LL_{\eta}^Kf.$
Then,
\begin{align*}\left|\E^{\e_n}_{\eta}(f_n^N,f_*)-\E_{\eta}^K(f_*,f_*)\right|
&\le \big|\langle f^N, (\LL_{\eta}^{n}-\LL_{\eta}^K)f_*\rangle_{L^2(\R^d;dx)}\big|+\big|\langle f^N_n-f^N,
\LL^{n}_{\eta}f_*\rangle_{L^2(\R^d;dx)}\big|\\
&\quad+\big|\langle f^N-f,\LL_{\eta}^Kf_*\rangle_{L^2(\R^d;dx)}\big|+\big|\langle f-f_*,\LL_{\eta}^Kf_*\rangle_{L^2(\R^d;dx)}\big|\\
&=:I_1+I_2+I_3+I_4.\end{align*} By the Cauchy-Schwarz inequality,
$$I_4\le \sqrt{\E_{\eta}^K(f-f_*,f-f_*)\E_{\eta}^K(f_*,f_*)}\le \gamma \sqrt{\E_{\eta}^K(f,f)+\gamma^2}.$$ Since $f_*\in C_c^\infty(\R^d)$, $\LL_{\eta}^Kf_*$ has a compact support. Due to \eqref{e:note-}, the argument for $I_{3,1}^{n,\eta}$ in the proof of Theorem \ref{T:2.2} yields that $\lim_{n\to\infty}I_2=0,$ thanks to (iii) in assumption {\bf (H)}. Similarly, using (iv) in assumption {\bf (H)} and following the argument for $I_{3,2}^{n,\eta}$ in the proof of Theorem \ref{T:2.2},  we can also verify that $\lim_{n\to
\infty}I_1=0.$ Obviously, by the Cauchy-Schwarz inequality again, $$I_3\le \|f^N-f\|_{L^2(\R^d;dx)}\|\LL_{\eta}^Kf_*\|_{L^2(\R^d;dx)}=\|f\I_{\{|f|>N\}}\|_{L^2(\R^d;dx)}\|\LL_{\eta}^Kf_*\|_{L^2(\R^d;dx)}.$$ Therefore, putting all
the estimates together, we find that
\begin{align*}\E_{\eta}^K(f_*,f_*)\le &\limsup_{n\to\infty}
\E^{\e_n}_{\eta}(f^N_n,f_*)+\gamma\sqrt{\E_{\eta}^K(f,f)+\gamma^2}+\|f\I_{\{|f|>N\}}\|_{L^2(\R^d;dx)}\|\LL_{\eta}^Kf_*\|_{L^2(\R^d;dx)}\\
\le&\sqrt{\limsup_{n\to\infty} \E^{\e_n}_{\eta}(f^N_n,f^N_n)}
\sqrt{\E_{\eta}^K(f_*,f_*)}+\gamma \sqrt{\E_{\eta}^K(f,f)+\gamma^2}+\|f\I_{\{|f|>N\}}\|_{L^2(\R^d;dx)}\|\LL_{\eta}^Kf_*\|_{L^2(\R^d;dx)}\\
\le&\sqrt{\limsup_{n\to\infty} \E^{\e_n}_{\eta}(f_n,f_n)}
\sqrt{\E_{\eta}^K(f_*,f_*)}+\gamma \sqrt{\E_{\eta}^K(f,f)+\gamma^2}+\|f\I_{\{|f|>N\}}\|_{L^2(\R^d;dx)}\|\LL_{\eta}^Kf_*\|_{L^2(\R^d;dx)},\end{align*} where in the second inequality we used the Cauchy-Schwarz inequality and (iv) in assumption {\bf (H)}. Letting $N\to \infty$,
we obtain
$$\E_{\eta}^K(f_*,f_*)\le\sqrt{\limsup_{n\to\infty} \E^{\e_n}_{\eta}(f_n,f_n)}
\sqrt{\E_{\eta}^K(f_*,f_*)}+\gamma \sqrt{\E_{\eta}^K(f,f)+\gamma^2}.$$
Note that
$$
\E_{\eta}^K(f_*,f_*)-\gamma\le \E_{\eta}^K(f,f)\le \E_{\eta}^K(f_*,f_*)+\gamma.
$$
Then, combining two inequalities above together, and
letting $\gamma\to0$ and $\eta\to0$,  $$\lim_{n\to\infty}\E^{\e_n}(f_n,f_n)\ge\E^K(f,f).$$ The proof is complete.
\end{proof}

We say that a sequence of bounded operators $\{T_n\}$ on $L^2(\R^d;\mu_{\e_n}(dx))$ \emph{strongly converges} to an operator $T$ on $L^2(\R^d;dx)$, if for every sequence $\{u_n\}_{n\ge1}$ with $u_n\in L^2(\R^d;\mu_{\e_n}(dx))$ for all $n\ge1$ and strongly converging
in $L^2$-spaces to $u\in L^2(\R^d;dx)$, the sequence $\{T_n u_n\}_{n\ge1}$ strongly converges in $L^2$-spaces to
$Tu$; see \cite[Definition 2.6]{KS} or \cite[Definition 2.4]{Ko2}.

Let $(P^{\e_n,\w}_t)_{t\ge0}$ and $(U_\lambda^{\e_n,\w})_{\lambda>0}$ be the semigroup  and the resolvent  associated with the Dirichlet form $ (\E^{\e_n,\w},\F^{\e_n,\w})$, respectively. Let $(P^{K}_t)_{t\ge0}$ and $(U_\lambda^{K})_{\lambda>0}$ be the semigroup  and the resolvent  associated with the Dirichlet form $ (\E^{K},\F^{K})$, respectively.
Then, by Theorem \ref{T:mos} and \cite[Theorem 2.4]{KS}, (also thanks to the fact that Theorem \ref{T:mos} holds for any $\{\e_n\}_{n\ge1}$ that converges to $0$ which is independent of $\w\in \Omega_0$), we have

\begin{corollary}\label{C:or}
Under assumption {\bf(H)},
there is $\Omega_0\subset \Omega$ of full probability measure  so that for every $\w \in \Omega_0$, $P^{\e,\w}_t$ strongly converges
 in $L^2$-spaces to $P^K_t$ for every $t>0$ as $\e\to 0$; equivalently, $U_\lambda^{\e,\w}$ strongly converges in $L^2$-spaces
 to $U^K_\lambda$ for every $\lambda>0$ as $\e\to 0$.
\end{corollary}

\bigskip

\noindent {\bf Acknowledgements.}
The research of Xin Chen is supported by the National Natural Science Foundation of China (No.\ 11871338).\
The research of Zhen-Qing Chen is  partially supported by Simons Foundation grant 520542 and a Victor Klee Faculty Fellowship at UW.\
The research of Takashi Kumagai is supported
by JSPS KAKENHI Grant Number JP17H01093 and by the Alexander von Humboldt Foundation.\
The research of Jian Wang is supported by the National
Natural Science Foundation of China (No.\ 11831014), the Program for Probability and Statistics: Theory and Application (No.\ IRTL1704) and the Program for Innovative Research Team in Science and Technology in Fujian Province University (IRTSTFJ).

\vskip 0.3truein
{\small
{\bf Xin Chen:}
   Department of Mathematics, Shanghai Jiao Tong University, 200240 Shanghai, P.R. China.
   \newline Email: \texttt{chenxin217@sjtu.edu.cn}
	
\bigskip
	
{\bf Zhen-Qing Chen:}
   Department of Mathematics, University of Washington, Seattle,
WA 98195, USA. \newline  Email: \texttt{zqchen@uw.edu}

\bigskip

{\bf Takashi Kumagai:}
 Research Institute for Mathematical Sciences,
Kyoto University, Kyoto 606-8502, Japan.
Email: \texttt{kumagai@kurims.kyoto-u.ac.jp}

\bigskip

{\bf Jian Wang:}
    College of Mathematics and Informatics \& Fujian Key Laboratory of Mathematical Analysis and Applications, Fujian Normal University, 350007 Fuzhou, P.R. China.
   Email:  \texttt{jianwang@fjnu.edu.cn}}

\end{document}